\documentclass{article}%
\usepackage{amsfonts}
\usepackage{amsmath}
\usepackage{amssymb}
\usepackage{graphicx}%
\providecommand{\U}[1]{\protect\rule{.1in}{.1in}}
\newtheorem{theorem}{Theorem}

\newtheorem{corollary}{Corollary}

\newtheorem{definition}{Definition}

\newtheorem{lemma}{Lemma}

\newtheorem{proposition}{Proposition}
\newtheorem{remark}{Remark}

\newenvironment{proof}[1][Proof]{\noindent\textbf{#1.} }{\ \rule{0.5em}{0.5em}}
\begin{document}

\title{Radially symmetric thin plate splines interpolating a circular contour map}
\author{Aurelian Bejancu\\Kuwait University, PO Box 5969, Safat 13060, Kuwait\thanks{E-mail:
\texttt{aurelian@sci.kuniv.edu.kw}}}
\date{February 6, 2015}
\maketitle

\begin{abstract}
Profiles of radially symmetric thin plate spline surfaces minimizing the Beppo
Levi energy over a compact annulus $R_{1}\leq r\leq R_{2}$ have been studied
by Rabut via reproducing kernel methods. Motivated by our recent construction
of Beppo Levi polyspline surfaces, we focus here on minimizing the radial
energy over the full semi-axis $0<r<\infty$. Using a $L$-spline approach, we
find two types of minimizing profiles: one is the limit of Rabut's solution as
$R_{1}\rightarrow0$ and $R_{2}\rightarrow\infty$ (identified as a
`non-singular' $L$-spline), the other has a second-derivative singularity and
matches an extra data value at $0$. For both profiles and $p\in\left[
2,\infty\right]  $, we establish the $L^{p}$-approximation order $3/2+1/p$ in
the radial energy space. We also include numerical examples and obtain a novel
representation of the minimizers in terms of dilates of a basis
function.\bigskip

\smallskip


\end{abstract}

\section{Introduction}

A well-known tool for scattered data interpolation, the \emph{thin plate
spline} surface is defined as the unique minimizer of the squared seminorm%
\begin{equation}
\left\Vert F\right\Vert _{BL}^{2}:=\iint\nolimits_{\mathbb{R}^{2}}\left(
\left\vert F_{xx}\right\vert ^{2}+2\left\vert F_{xy}\right\vert ^{2}%
+\left\vert F_{yy}\right\vert ^{2}\right)  \mathrm{d}x\,\mathrm{d}y,
\label{eq:Duchon}%
\end{equation}
among all admissible $F$ taking prescribed values at given points (not all
collinear) in the plane. Here, $BL$ denotes the \emph{Beppo Levi}
space of admissible continuous functions $F$ with generalized
second-order derivatives in $L^{2}\left(  \mathbb{R}^{2}\right)  $, 
equipped with the above defined seminorm. Originally
discovered and utilized in aerospace engineering (Sabin \cite{MAS69}, Harder
and Desmarais \cite{HD72}), thin plate splines became a research topic in
approximation theory with the work of Duchon \cite{Du76}.

A different bivariate interpolation problem is to construct a surface that matches
boundary and internal curves within a bounded domain---in CAGD jargon, this is 
referred to as \emph{transfinite} interpolation. One solution to this problem, 
via Kounchev's  \emph{polyspline} method \cite{Ku01}, defines a unique 
piecewise biharmonic and globally $C^{2}$ surface passing through the given 
curve data, subject to additional boundary conditions, such as given normal 
derivative values.

Using the polyspline method, Bejancu \cite{MMA12} has recently proposed a 
new type of thin plate spline surfaces, matching continuous curves
prescribed along concentric circles. These surfaces, called 
\emph{Beppo Levi polysplines on annuli},
are minimizers of the polar coordinate version of (\ref{eq:Duchon}), namely%
\begin{equation}
\left\Vert f\right\Vert _{BL}^{2}:=\int_{0}^{\infty}\int_{-\pi}^{\pi}\left\{
\left\vert f_{rr}\right\vert ^{2}+2\left\vert \frac{f_{\theta}}{r^{2}}%
-\frac{f_{\theta r}}{r}\right\vert ^{2}+\left\vert \frac{f_{\theta\theta}%
}{r^{2}}+\frac{f_{r}}{r}\right\vert ^{2}\right\}  r\,\mathrm{d}\theta
\,\mathrm{d}r, \label{eq:Du-polar}%
\end{equation}
where $f\left(  r,\theta\right)  :=F\left(  r\cos\theta,r\sin\theta\right)  $
is the polar form of $F\in C^{2}\left(  \mathbb{R}^{2}\backslash\left\{
0\right\}  \right)  \cap BL$. Related Beppo Levi polyspline surfaces
interpolating continuous periodic data along parallel lines or hyperplanes
have also been studied in \cite{JAT08,CA11,IMA13}.

The construction of Beppo Levi polysplines on annuli is based on their Fourier
series representation in $\theta$, with amplitude coefficients depending on
$r$. In this paper, we focus on the special case of radially symmetric Beppo
Levi polysplines, which arise as zero-frequency amplitudes (non-zero
frequencies are treated in \cite{L-spl}). Note that, for radial functions
$f\left(  r,\theta\right)  \equiv f\left(  r\right)  $, the energy functional
(\ref{eq:Du-polar}) takes the following form (up to a constant factor):%
\begin{equation}
\int_{0}^{\infty}\left(  r\,\left\vert \dfrac{\mathrm{d}^{2}f}{\mathrm{d}%
r^{2}}\right\vert ^{2}+\frac{1}{r}\left\vert \dfrac{\mathrm{d}f}{\mathrm{d}%
r}\right\vert ^{2}\right)  \mathrm{d}r. \label{eq:semin0}%
\end{equation}
Therefore, minimizing (\ref{eq:semin0}) subject to $f$ taking prescribed
values at the set of positive radii $r_{1}<\ldots<r_{n}$ is equivalent to
determining the profile $f$ of a \emph{radially symmetric thin plate spline}
surface interpolating a circular contour map at the concentric circles
$r=r_{j}$, $1\leq j\leq n$.

For $0<R_{1}<R_{2}<\infty$, Rabut \cite{Rab96} has previously studied the
related problem of minimizing $%
%TCIMACRO{\tint _{R_{1}}^{R_{2}}}%
%BeginExpansion
{\textstyle\int_{R_{1}}^{R_{2}}}
%EndExpansion
(\,r\left\vert f^{\prime\prime}\right\vert ^{2}+\frac{1}{r}\left\vert
f^{\prime}\right\vert ^{2}\,)\,\mathrm{d}r$ in place of (\ref{eq:semin0}),
subject to the interpolation conditions at $r_{1}$, \ldots, $r_{n}$. Rabut
obtained a closed form solution to this problem via the method of abstract
splines and reproducing kernels.

Our approach (section 2) exploits the fact that an interpolating minimizer of
(\ref{eq:semin0}) is actually a univariate $L$-spline over $\left(
0,\infty\right)  $, of continuity class $C^{2}$ at the knots 
$r_{1}$, \ldots, $r_{n}$. The key observation is
that, after imposing the interpolation constraints and the condition that
(\ref{eq:semin0}) be finite, there remains an extra degree of freedom on the
interval $\left(  0,r_{1}\right)  $. This enables us to obtain two types of
minimizers: one is the limit of Rabut's profile as $R_{1}\rightarrow0$ and
$R_{2}\rightarrow\infty$ (generating a surface which is biharmonic 
at $0$), the other has a second-derivative singularity and
matches an extra data value at $0$. We study the main properties of these
profiles in section 3, including their linear representation over $\left(
0,\infty\right)  $ in terms of dilates of a basis function.

For both minimizing profiles and $p\in\left[  2,\infty\right]  $, in section 4
we establish a $L^{p}\left[  r_{1},r_{n}\right]  $-error bound of order
$O\left(  h^{3/2+1/p}\right)  $, where $h$ is the maximum distance between
consecutive radii. This is further applied to derive, in \cite{L-spl}, a
$L^{2}$-convergence result for transfinite surface interpolation with
biharmonic Beppo Levi polysplines on annuli. We also prove that the exponent
$3/2$ in the above approximation order cannot be increased in general for data
functions from the radial energy space. In section 5, we illustrate the
numerical accuracy, as well as several graphs of the resulting interpolatory 
$L$-spline profiles, and point out a connection with Johnson's recent 
construction of compactly supported radial basis functions \cite{Jo12}.

\section{Preliminaries}

\subsection{Admissible profiles}

Let $AC_{loc}$ be the vector space of functions $f:\left(  0,\infty\right)
\rightarrow\mathbb{C}$ which are absolutely continuous on any interval
$\left[  a,b\right]  $, where $0<a<b<\infty$. Throughout this paper,
$\Lambda_{0}$ denotes the space of complex-valued functions $f\in
C^{1}\left(  0,\infty\right)  $, such that $f^{\prime}\in AC_{loc}$ and the
integral (\ref{eq:semin0}) is finite. For any $f$, $g\in\Lambda_{0}$, we
define the semi-inner product%
\begin{equation}
\left\langle f,g\right\rangle _{0}:=%
%TCIMACRO{\dint _{0}^{\infty}}%
%BeginExpansion
{\displaystyle\int_{0}^{\infty}}
%EndExpansion
\left(  r\,f^{\prime\prime}\left(  r\right)  \overline{g}^{\prime\prime
}\left(  r\right)  +\dfrac{1}{r}\,f^{\prime}\left(  r\right)  \overline
{g}^{\prime}\left(  r\right)  \right)  \mathrm{d}r. \label{eq:semi-prod}%
\end{equation}
The induced squared seminorm $\left\Vert f\right\Vert _{0}^{2}:=\left\langle
f,f\right\rangle _{0}$ on $\Lambda_{0}$ equals the radial Beppo Levi energy
(\ref{eq:semin0}).

\begin{lemma}
\label{le:aux0}If $f\in\Lambda_{0}$, then $f$ can be extended by continuity at
$0$. Also, the relation $f^{\prime}\left(  r\right)  =O\left(  1\right)  $
holds as $r\rightarrow0$ and as $r\rightarrow\infty$.
\end{lemma}

\begin{proof}
Using the Leibniz-Newton formula $f\left(  r\right)  =f\left(  s\right)
+\int_{s}^{r}f^{\prime}\left(  t\right)  \mathrm{d}t$ for $r,s>0$, as well as
the Cauchy-Schwarz estimate%
\begin{align}
\left\vert f\left(  r\right)  -f\left(  s\right)  \right\vert  &  =\left\vert
\int_{s}^{r}t^{1/2}\left[  t^{-1/2}f^{\prime}\left(  t\right)  \right]
\mathrm{d}t\right\vert \nonumber\\
&  \leq\left\vert \int_{s}^{r}t\,\mathrm{d}t\right\vert ^{1/2}\left\vert
\int_{s}^{r}\left\vert t^{-1/2}f^{\prime}\left(  t\right)  \right\vert
^{2}\mathrm{d}t\right\vert ^{1/2}\label{eq:aux1}\\
&  \leq2^{-1/2}\left\vert r^{2}-s^{2}\right\vert ^{1/2}\left\Vert f\right\Vert
_{0},\nonumber
\end{align}
it follows that $f$ is Lipschitz of order $1/2$ on any bounded subinterval of
$\left(  0,\infty\right)  $. In particular, $f$ is uniformly continuous on
$\left(  0,1\right)  $, hence it can be extended by continuity at $0$. Next,
since $f^{\prime}\in AC_{loc}$, we have, for each $r>0$,%
\begin{align*}
r^{-1}f^{\prime}\left(  r\right)  -f^{\prime}\left(  1\right)   &  =\int
_{1}^{r}\left[  t^{-1}f^{\prime}\left(  t\right)  \right]  ^{\prime}%
\mathrm{d}t\\
&  =\int_{1}^{r}t^{-3/2}\left[  t^{1/2}f^{\prime\prime}\left(  t\right)
-t^{-1/2}f^{\prime}\left(  t\right)  \right]  \mathrm{d}t.
\end{align*}
Therefore the estimate%
\begin{eqnarray*}
\lefteqn{  \left\vert \int_{1}^{r}t^{-3/2}\left[  t^{1/2}f^{\prime\prime}\left(
t\right)  \right]  \mathrm{d}t\right\vert +\left\vert \int_{1}^{r}%
t^{-3/2}\left[  t^{-1/2}f^{\prime}\left(  t\right)  \right]  \mathrm{d}%
t\right\vert } \\
&  \leq & \left\vert \int_{1}^{r}t^{-3}\mathrm{d}t\right\vert ^{1/2}\left\{
\left\vert \int_{1}^{r}\left\vert t^{1/2}f^{\prime\prime}\left(  t\right)
\right\vert ^{2}\mathrm{d}t\right\vert ^{1/2}+\left\vert \int_{1}%
^{r}\left\vert t^{-1/2}f^{\prime}\left(  t\right)  \right\vert ^{2}%
\mathrm{d}t\right\vert ^{1/2}\right\} \\
&  \leq & 2^{1/2}\left\vert 1-r^{-2}\right\vert ^{1/2}\left\Vert f\right\Vert
_{0}%
\end{eqnarray*}
implies the existence of a constant $C_{f}\geq0$ such that%
\[
\left\vert f^{\prime}\left(  r\right)  \right\vert \leq C_{f}\left(
r+\left\vert 1-r^{2}\right\vert ^{1/2}\right)  ,\quad\forall r>0,
\]
hence $f^{\prime}\left(  r\right)  =O\left(  1\right)  $ as $r\rightarrow0$.
The boundedness of $f^{\prime}$ at $\infty$ follows similarly from%
\begin{eqnarray*}
\left\vert  rf^{\prime}\left(  r\right)  -f^{\prime}\left(  1\right) \right\vert
&  = & \left\vert  \int_{1}^{r}\left[  
tf^{\prime}\left(  t\right)  \right]  ^{\prime}\mathrm{d}t \right\vert \\
&  = & \left\vert \int_{1}^{r}t^{1/2}\left[  t^{1/2}f^{\prime\prime}\left(  t\right)
+t^{-1/2}f^{\prime}\left(  t\right)  \right]  \mathrm{d}t \right\vert \\
& \leq &
\left\vert \int_{1}^{r}t^{1/2}\left[  t^{1/2}f^{\prime\prime}\left(
t\right)  \right]  \mathrm{d}t\right\vert +\left\vert \int_{1}^{r}%
t^{1/2}\left[  t^{-1/2}f^{\prime}\left(  t\right)  \right]  
\mathrm{d}t\right\vert,
\end{eqnarray*}
followed by the Cauchy-Schwarz estimate:
\begin{eqnarray*}
\lefteqn{ \left\vert  rf^{\prime}\left(  r\right)  -f^{\prime}\left(  1\right) 
\right\vert  } \\
&  \leq & \left\vert \int_{1}^{r}t\mathrm{d}t\right\vert ^{1/2}\left\{
\left\vert \int_{1}^{r}\left\vert t^{1/2}f^{\prime\prime}\left(  t\right)
\right\vert ^{2}\mathrm{d}t\right\vert ^{1/2}+\left\vert \int_{1}%
^{r}\left\vert t^{-1/2}f^{\prime}\left(  t\right)  \right\vert ^{2}%
\mathrm{d}t\right\vert ^{1/2}\right\} \\
&  \leq & 2^{-1/2}\left\vert r^{2}-1\right\vert ^{1/2}\left\Vert f\right\Vert
_{0}.
\end{eqnarray*}
The proof is complete.
\end{proof}

\subsection{$L$-spline framework}

\label{sec2}Our $L$-spline approach is motivated by the necessary conditions
satisfied by a minimizer of (\ref{eq:semin0}). First, note that the
self-adjoint Euler-Lagrange differential operator associated to the integrand
of (\ref{eq:semin0}) is%
\begin{align*}
L_{0}  &  :=\dfrac{\mathrm{d}^{2}}{\mathrm{d}r^{2}}\left(  r\dfrac
{\mathrm{d}^{2}}{\mathrm{d}r^{2}}\right)  -\dfrac{\mathrm{d}}{\mathrm{d}%
r}\left(  \dfrac{1}{r}\dfrac{\mathrm{d}}{\mathrm{d}r}\right) \\
&  =r\dfrac{\mathrm{d}^{4}}{\mathrm{d}r^{4}}+2\dfrac{\mathrm{d}^{3}%
}{\mathrm{d}r^{3}}-\dfrac{1}{r}\dfrac{\mathrm{d}^{2}}{\mathrm{d}r^{2}}%
+\dfrac{1}{r^{2}}\dfrac{\mathrm{d}}{\mathrm{d}r}\\
&  =r\left(  \dfrac{\mathrm{d}^{2}}{\mathrm{d}r^{2}}+\dfrac{1}{r}%
\dfrac{\mathrm{d}}{\mathrm{d}r}\right)  ^{2}.
\end{align*}
The substitution $r=e^{v}$, $\frac{\mathrm{d}}{\mathrm{d}v}=r\frac{\mathrm{d}%
}{\mathrm{d}r}$, changes $L_{0}$ into a differential operator with constant
coefficients in variable $v$. Factoring this operator and reverting to
variable $r$, we obtain%
\[
L_{0}=\dfrac{1}{r^{3}}\left(  r\dfrac{\mathrm{d}}{\mathrm{d}r}-2\right)
^{2}\left(  r\dfrac{\mathrm{d}}{\mathrm{d}r}\right)  ^{2},
\]
and we identify the null space of $L_{0}$ as the $4$-dimensional vector space%
\begin{equation}
\mathrm{Ker}\,L_{0}=\mathrm{span}\left\{  r^{2},r^{2}\ln r,1,\ln r\right\}  .
\label{eq:ker}%
\end{equation}
Hence usual calculus of variations suggests that a minimizer of
(\ref{eq:semin0}) subject to prescribed values at $r_{1}<\ldots<r_{n}$ should
be in $\mathrm{Ker}\,L_{0}$ on each of the open intervals $\left(
0,r_{1}\right)  $, $\left(  r_{1},r_{2}\right)  $, \ldots, $\left(
r_{n},\infty\right)  $.

Second, as observed in \cite{MSc}, the convergence of the radial Beppo Levi
energy integral (\ref{eq:semin0}) at $0$ and $\infty$ implies that a minimizer
should in fact belong to $\mathrm{span}\left\{  r^{2},r^{2}\ln r,1\right\}  $
for $r\in\left(  0,r_{1}\right)  $, and to $\mathrm{span}\left\{  1,\ln
r\right\}  $ for $r\in\left(  r_{n},\infty\right)  $. These two spans are null
spaces of the following left and right `boundary' operators:%
\[%
\begin{array}
[c]{l}%
G_{0}:=\dfrac{\mathrm{d}^{3}}{\mathrm{d}r^{3}}-\dfrac{1}{r}\dfrac
{\mathrm{d}^{2}}{\mathrm{d}r^{2}}+\dfrac{1}{r^{2}}\dfrac{\mathrm{d}%
}{\mathrm{d}r}=\dfrac{1}{r^{3}}\left(  r\dfrac{\mathrm{d}}{\mathrm{d}%
r}-2\right)  ^{2}\left(  r\dfrac{\mathrm{d}}{\mathrm{d}r}\right)  ,\\
\\
R_{0}:=\dfrac{1}{r}\left[  \dfrac{\mathrm{d}^{2}}{\mathrm{d}r^{2}}+\dfrac
{1}{r}\dfrac{\mathrm{d}}{\mathrm{d}r}\right]  =\dfrac{1}{r^{3}}\left(
r\dfrac{\mathrm{d}}{\mathrm{d}r}\right)  ^{2}.
\end{array}
\]

Let $\rho$ be the given set of positive knots $r_{1}<\ldots<r_{n}$. 
The above considerations support the following definition.

\begin{definition}
\label{def:natLspline}\emph{(a)} A function $\eta:\left[  0,\infty\right)
\rightarrow\mathbb{C}$ is called a \emph{Beppo Levi }$L_{0}$\emph{-spline on
}$\rho$ if the following conditions hold:

\emph{(i)} $L_{0}\eta\left(  r\right)  =0,\;\forall r\in\left(  r_{j}%
,r_{j+1}\right)  ,\;\forall j\in\left\{  1,\ldots,n-1\right\}  $;

\emph{(ii)} $G_{0}\eta\left(  r\right)  =0$, $\forall r\in\left(
0,r_{1}\right)  $, and $R_{0}\eta\left(  r\right)  =0$, $\forall r>r_{n}$;

\emph{(iii)} $\eta$ is $C^{2}$-continuous at each knot $r_{1}$, \ldots,
$r_{n}$.

We denote by $\mathcal{S}_{0}\left(  \rho\right)  \subset\Lambda_{0}$ the
class of all Beppo Levi $L_{0}$-splines on $\rho$.

\noindent\emph{(b)} A Beppo Levi $L_{0}$-spline $\eta\in\mathcal{S}_{0}\left(
\rho\right)  $ that satisfies $\eta\in\mathrm{span}\left\{  r^{2},1\right\}  $
for $r\in\left(  0,r_{1}\right)  $ is called \emph{non-singular}.
\end{definition}

\begin{remark}
\emph{On any interval of positive real numbers, Kounchev \cite[p.\ 104]{Ku01}
characterized the null space of }$L_{0}$ \emph{as an extended complete
Chebyshev (ECT) system in the sense of Karlin and Ziegler \cite{KZ66}, due to
the representation}%
\[
L_{0}=D_{4}D_{3}D_{2}D_{1}=\dfrac{\mathrm{d}}{\mathrm{d}r}r\dfrac{\mathrm{d}%
}{\mathrm{d}r}\frac{1}{r}\dfrac{\mathrm{d}}{\mathrm{d}r}r\dfrac{\mathrm{d}%
}{\mathrm{d}r},
\]
\emph{where }$D_{1}=\frac{\mathrm{d}}{\mathrm{d}r}$\emph{, }$D_{2}=D_{4}%
=\frac{\mathrm{d}}{\mathrm{d}r}\left(  r\cdot\right)  $\emph{, }$D_{3}%
=\frac{\mathrm{d}}{\mathrm{d}r}\left(  \frac{1}{r}\cdot\right)  $\emph{.
Following Schumaker \cite[p.\ 398]{LLS07}, we note that }$L_{0}$\emph{ also
admits the factorization}%
\[
L_{0}=L^{\ast}L,
\]
\emph{where }$L:=\frac{1}{\sqrt{r}}\frac{\mathrm{d}}{\mathrm{d}r}%
r\frac{\mathrm{d}}{\mathrm{d}r}=r^{3/2}R_{0}$ \emph{and }$L^{\ast}$\emph{
denotes the formal adjoint of }$L$\emph{. This type of factorization was
employed in the early studies of Ahlberg, Nilson, and Walsh \cite{ANW67} and
Schultz and Varga \cite{SV67} to define `generalized splines' and `}%
$L$\emph{-splines' as functions that are piecewise in the null space of
}$L^{\ast}L$\emph{ and satisfy certain continuity conditions. Our definition
adopts the subsequent terminology of Lucas \cite{L70} and Jerome and Pierce
\cite{JP72}, according to which `}$L$\emph{-splines' are piecewise in the null
space of a general self-adjoint differential operator }$L$\emph{ with variable
coefficients.}
\end{remark}

\begin{remark}
\emph{The Beppo Levi boundary conditions (ii) of our definition do not agree
with the usual `natural' (or `type II') conditions from }$L$\emph{-spline
literature \cite{L70,JP72,LLS07}, the latter being formulated in terms of one
and the same differential operator outside the interpolation domain or at the
endpoints of this domain.}
\end{remark}

\section{Interpolation with Beppo Levi $L_{0}$-splines}

Due to the new type of boundary conditions, the main properties of
interpolation with Beppo Levi $L_{0}$-splines are not direct consequences of
classical $L$-spline theory, but will be established in this section based on
our first theorem below.

\subsection{A fundamental orthogonality result}

\begin{theorem}
\label{thm:FI}Let $\eta\in\mathcal{S}_{0}\left(  \rho\right)  $ and $\psi
\in\Lambda_{0}$ such that%
\begin{equation}
\psi\left(  r_{j}\right)  =0,\quad\forall j\in\left\{  1,\ldots,n\right\}  .
\label{eq:zero-val}%
\end{equation}
If, in addition, we assume that either $\psi$ satisfies%
\begin{equation}
\psi\left(  0\right)  =0, \label{eq:zero-lim}%
\end{equation}
or $\eta$ is non-singular, then%
\begin{equation}
\left\langle \eta,\psi\right\rangle _{0}=0. \label{eq:ortho}%
\end{equation}

\end{theorem}

\begin{proof}
Using the notation $r_{0}:=0$, $r_{n+1}:=\infty$, (\ref{eq:semi-prod})
implies%
\[
\left\langle \eta,\psi\right\rangle _{0}=\sum_{j=1}^{n+1}\int_{r_{j-1}}%
^{r_{j}}\left[  r\,\eta^{\prime\prime}\left(  r\right)  \overline{\psi
}^{\prime\prime}\left(  r\right)  +\frac{1}{r}\eta^{\prime}\left(  r\right)
\overline{\psi}^{\prime}\left(  r\right)  \right]  \mathrm{d}r.
\]
Since $\psi^{\prime}\in AC_{loc}$, we may apply integration by parts to the
first term of each integral to obtain%
\begin{eqnarray*}
\left\langle \eta,\psi\right\rangle _{0}  &  = & \sum_{j=1}^{n+1}\left[
r\,\eta^{\prime\prime}\left(  r\right)  \overline{\psi}^{\prime}\left(
r\right)  \right]  _{r_{j-1}}^{r_{j}}\\
& & \mbox{} -\sum_{j=1}^{n+1}\int_{r_{j-1}}^{r_{j}}
\overline{\psi}^{\prime}(r)  \left[  r\,\eta^{\prime\prime\prime} (r)  
+ \eta^{\prime\prime}(r)  -\frac{1}{r}\eta^{\prime}(r)
\right]  \mathrm{d}r.
\end{eqnarray*}
Note that the first sum of the above right-hand side is telescopic due to the
continuity of $\eta^{\prime\prime}$. Hence, we only need to evaluate the
boundary terms of this sum corresponding to $r:=r_{0}=0$ and $r:=r_{n+1}%
=\infty$. We invoke the specific form of $\eta\in\mathcal{S}_{0}\left(
\rho\right)  $ on the extreme sub-intervals:$\ \eta\in\mathrm{span}\left\{
r^{2},r^{2}\ln r,1\right\}  $ for $r\in\left(  0,r_{1}\right)  $, and $\eta
\in\mathrm{span}\left\{  1,\ln r\right\}  $ for $r\in\left(  r_{n}%
,\infty\right)  $. By Lemma~\ref{le:aux0}, it follows that%
\[
r\,\eta^{\prime\prime}\left(  r\right)  \overline{\psi}^{\prime}\left(
r\right)  =\left\{
\begin{array}
[c]{l}%
O\left(  r\left\vert \ln r\right\vert \right)  ,\quad\text{as }r\rightarrow
0,\\
O\left(  r^{-1}\right)  ,\quad\text{as }r\rightarrow\infty.
\end{array}
\right.
\]
Therefore the boundary terms vanish in the limit at $0$ and $\infty$.

On the other hand, the fact that $\eta\in\mathrm{Ker}\,L_{0}$ on each
sub-interval implies the existence, for each $j\in\left\{  1,\ldots
,n+1\right\}  $, of a constant $b_{j}$ such that%
\[
r\,\eta^{\prime\prime\prime}\left(  r\right)  +\eta^{\prime\prime}\left(
r\right)  -\frac{1}{r}\eta^{\prime}\left(  r\right)  =b_{j},\quad\forall
r\in\left(  r_{j-1},r_{j}\right)  ,\ \forall j\in\left\{  1,\ldots
,n+1\right\}  ,
\]
where $b_{n+1}=0$, since $\eta\left(  r\right)  \in\mathrm{span}\left\{  1,\ln
r\right\}  $ for $r\in\left(  r_{n},\infty\right)  $. Hence%
\[
\left\langle \eta,\psi\right\rangle _{0}=-\sum_{j=1}^{n}b_{j}\int_{r_{j-1}%
}^{r_{j}}\overline{\psi}^{\prime}\left(  r\right)  \mathrm{d}r=-\sum_{j=1}%
^{n}b_{j}\left[  \overline{\psi}\left(  r\right)  \right]  _{r_{j-1}}^{r_{j}%
}=b_{1}  \overline{\psi}\left(  0\right)   ,
\]
the last equality being due to the vanishing condition (\ref{eq:zero-val}) of
$\psi$ at the positive knots. Therefore (\ref{eq:zero-lim}) implies
$\left\langle \eta,\psi\right\rangle _{0}=0$, as required. The same conclusion
applies assuming that $\eta$ is non-singular, since $b_{1}=0$ in that case.
\end{proof}

\subsection{Existence, uniqueness, variational characterization}

Our $L$-spline formulation shows that the extra degree of freedom of any
$\eta\in\mathcal{S}_{0}\left(  \rho\right)  $ on the leftmost interval
$\left(  0,r_{1}\right)  $ allows one more restriction to be imposed on $\eta
$, apart from the interpolation conditions at the knots. The next result
obtains existence and uniqueness for two types of interpolatory profiles from
$\mathcal{S}_{0}\left(  \rho\right)  $: one matching an extra data value at
$0$, the other being non-singular.

\begin{theorem}
\label{thm:EU}Let $\alpha$, $\nu_{1}$, $\nu_{2}$,\ldots, $\nu_{n}$ be
arbitrary real values.

\emph{(a)} There exists a unique Beppo Levi $L_{0}$-spline $\sigma^{A}%
\in\mathcal{S}_{0}\left(  \rho\right)  $, such that%
\begin{equation}
\left\{
\begin{array}
[c]{l}%
\sigma^{A}\left(  r_{j}\right)  =\nu_{j},\quad\forall j\in\left\{
1,\ldots,n\right\}  ,\\
\sigma^{A}\left(  0\right)  =\alpha.
\end{array}
\right.  \label{eq:typeA}%
\end{equation}

\emph{(b)} There exists a unique non-singular Beppo Levi $L_{0}$-spline
$\sigma^{B}\in\mathcal{S}_{0}\left(  \rho\right)  $, such that%
\begin{equation}
\sigma^{B}\left(  r_{j}\right)  =\nu_{j},\quad\forall j\in\left\{
1,\ldots,n\right\}  . \label{eq:typeB}%
\end{equation}

\end{theorem}

\begin{proof}
For convenience, we first prove part (b). It is sufficient to establish the
existence of a unique function $\widetilde{\sigma}:\left[  r_{1},r_{n}\right]
\longrightarrow\mathbb{C}$ with the properties: i) $\widetilde{\sigma}%
\in\mathrm{Ker}L_{0}$ on $\left(  r_{j-1},r_{j}\right)  $ for $j\in\left\{
2,\ldots,n\right\}  $; ii) $\widetilde{\sigma}\in C^{2}\left(  r_{1}%
,r_{n}\right)  $; iii) the interpolation conditions (\ref{eq:typeB}) hold for
$\widetilde{\sigma}$ in place of $\sigma^{B}$; and iv) $\widetilde{\sigma}$
satisfies the endpoint conditions:%
\begin{equation}
\left\{
\begin{array}
[c]{l}%
\left[  \left(  r\frac{\mathrm{d}}{\mathrm{d}r}\right)  \left(  r\frac
{\mathrm{d}}{\mathrm{d}r}-2\right)  \widetilde{\sigma}\left(  r\right)
\right]  _{r\rightarrow r_{1}^{+}}=0,\\
\left[  \left(  r\frac{\mathrm{d}}{\mathrm{d}r}\right)  ^{2}\widetilde{\sigma
}\left(  r\right)  \right]  _{r\rightarrow r_{n}^{-}}=0.
\end{array}
\right.  \label{eq:endpt}%
\end{equation}
Indeed, one can uniquely determine the constants $c_{m}$ for $m\in\left\{
1,2,3,4\right\}  $ such that the function $\sigma^{B}$ defined by%
\[
\sigma^{B}\left(  r\right)  :=\left\{
\begin{array}
[c]{ll}%
c_{1}r^{2}+c_{2}, & \text{if }0<r<r_{1},\\
\widetilde{\sigma}\left(  r\right)  , & \text{if }r_{1}\leq r\leq r_{n},\\
c_{3}+c_{4}\ln r, & \text{if }r_{n}<r,
\end{array}
\right.
\]
is continuous and has a continuous first derivative at $r_{1}$ and $r_{n}$.
This also ensures that $\frac{\mathrm{d}^{2}}{\mathrm{d}r^{2}}\sigma^{B}$ is
continuous at $r_{1}$ and $r_{n}$, due to (\ref{eq:endpt}) and the fact that
$\left(  r\frac{\mathrm{d}}{\mathrm{d}r}\right)  \left(  r\frac{\mathrm{d}%
}{\mathrm{d}r}-2\right)  \sigma^{B}\left(  r\right)  =0$, $\forall0<r<r_{1}$,
and $R_{0}\sigma^{B}\left(  r\right)  =0$, $\forall r>r_{n}$. Hence,
$\sigma^{B}$ verifies the conditions required by the conclusion. Conversely,
the above properties of $\widetilde{\sigma}$ hold with necessity for the
restriction to $\left[  r_{1},r_{N}\right]  $ of any non-singular Beppo Levi
$L_{0}$-spline $\sigma^{B}\in\mathcal{S}_{0}\left(  \rho\right)  $ that
satisfies (\ref{eq:typeB}).

Note that a function $\widetilde{\sigma}$ as described in the previous
paragraph is determined by four coefficients on each subinterval $\left(
r_{j-1},r_{j}\right)  $, $j\in\left\{  2,\ldots,n\right\}  $. These
coefficients are required to satisfy the homogeneous linear equations given by
three $C^{2}$-continuity conditions at each interior knot $r_{2}%
,\ldots,r_{n-1}$ and two endpoint conditions (\ref{eq:endpt}), as well as the
$n$ interpolation conditions (\ref{eq:typeB}). Hence this $4\left(
n-1\right)  \times4\left(  n-1\right)  $ system of linear equations becomes
homogeneous if we assume zero interpolation values: $\nu_{j}=0$, $j\in\left\{
1,\ldots,n\right\}  $. Let $\widetilde{\sigma}$ be determined by an arbitrary
solution of this homogeneous system and let $\sigma^{B}\in\mathcal{S}%
_{0}\left(  \rho\right)  $ be the unique extension of $\widetilde{\sigma}$ to
a non-singular Beppo Levi $L_{0}$-spline as in the previous paragraph.
Choosing $\eta=\psi:=\sigma^{B}$ in Theorem~\ref{thm:FI}, the orthogonality
relation (\ref{eq:ortho}) implies $\left\Vert \sigma^{B}\right\Vert _{0}=0$,
hence $\sigma^{B}\ $is a constant function on $\left(  0,\infty\right)  $.
Since $\sigma^{B}\left(  r_{j}\right)  =0$, $j\in\left\{  1,\ldots,n\right\}
$, we obtain $\sigma^{B}\equiv0$, which shows that the homogeneous linear
system admits only the trivial solution. Therefore the $4\left(  n-1\right)
\times4\left(  n-1\right)  $ non-homogeneous system corresponding to arbitrary
interpolation values has a unique solution, as required.

Part (a) of the theorem is similarly reduced to the existence of a unique
function $\widetilde{\sigma}$, this time defined on $\left[  0,r_{n}\right]
$, with the properties: i) $\widetilde{\sigma}\in\mathrm{Ker}L_{0}$ on
$\left(  r_{j-1},r_{j}\right)  $ for $j\in\left\{  2,\ldots,n\right\}  $,
while $\widetilde{\sigma}\in\mathrm{span}\left\{  r^{2},r^{2}\ln r,1\right\}
$ on $[0,r_{1})$; ii) $\widetilde{\sigma}\in C^{2}\left(  0,r_{n}\right)  $;
iii) $\widetilde{\sigma}$ satisfies (\ref{eq:typeA}) (in place of $\sigma^{A}%
$); iv) $\widetilde{\sigma}$ satisfies the second of the endpoint conditions
(\ref{eq:endpt}). All these constraints amount to a system of $4n-1$ linear
equations for as many coefficients. Then arguments similar to those of the
last paragraph show that the corresponding homogeneous linear system has only
the trivial solution, which completes the proof.
\end{proof}

\begin{remark}
\emph{Existence and uniqueness results for }$L_{0}$\emph{-spline interpolation
have previously been obtained by Kounchev \cite{Ku98,Ku01} under different
boundary conditions. Specifically, Kounchev uses a clamped condition at the
right-end point }$r_{n}$\emph{ (}i.e.\emph{, the first-derivative is
prescribed at }$r_{n}$\emph{), while, at the left-end point }$r_{1}$\emph{,
his }$L_{0}$\emph{-spline either is clamped or it satisfies the
non-singularity condition stated in part (b) of
Definition~\ref{def:natLspline}.}
\end{remark}

The next result shows that our $L_{0}$-spline interpolants are indeed profiles
of radially symmetric thin plate spline surfaces minimizing the radial Beppo
Levi seminorm (\ref{eq:semin0}).

\begin{theorem}
\label{thm:var-char}Given arbitrary real values $\alpha$, $\nu_{1}$, $\nu_{2}%
$,\ldots, $\nu_{n}$, let $\sigma^{A}$ and $\sigma^{B}$ denote the unique Beppo
Levi $L_{0}$-splines obtained in Theorem~\ref{thm:EU}. Then $\left\Vert
\sigma^{A}\right\Vert _{0}<\left\Vert f\right\Vert _{0}$ whenever $f\in
\Lambda_{0}$ satisfies \emph{(\ref{eq:typeA})} in place of $\sigma^{A}$ and
$f\not =\sigma^{A}$. Also, if $f\in\Lambda_{0}$, $f\not =\sigma^{B}$, and $f$
satisfies conditions \emph{(\ref{eq:typeB})} in place of $\sigma^{B}$, then
$\left\Vert \sigma^{B}\right\Vert _{0}<\left\Vert f\right\Vert _{0}$.
\end{theorem}

\begin{proof}
Assume that $f\in\Lambda_{0}$, $f$ satisfies the same interpolation conditions
(\ref{eq:typeA}) as $\sigma^{A}$, and let $\eta:=\sigma^{A}$, $\psi
:=f-\sigma^{A}$. Since $\psi$ satisfies (\ref{eq:zero-val}) and
(\ref{eq:zero-lim}), by Theorem~\ref{thm:FI} we have%
\[
\left\langle \sigma^{A},f-\sigma^{A}\right\rangle _{0}=0,
\]
which implies the \emph{first integral relation}%
\begin{equation}
\left\Vert f\right\Vert _{0}^{2}=\left\Vert \sigma^{A}\right\Vert _{0}%
^{2}+\left\Vert f-\sigma^{A}\right\Vert _{0}^{2}. \label{eq:pythagoras}%
\end{equation}
Therefore $\left\Vert f\right\Vert _{0}\geq\left\Vert \sigma^{A}\right\Vert
_{0}$, with equality only if $\left\Vert f-\sigma^{A}\right\Vert _{0}=0$,
which is equivalent to $\left(  f-\sigma^{A}\right)  ^{\prime}\left(
r\right)  =0$, for all $r\in\left(  0,\infty\right)  $. Since $f-\sigma^{A}$
takes zero values at the knots $r_{1}$,\ldots, $r_{n}$, this implies
$f\equiv\sigma^{A}$, which proves the first part of the theorem. The second
part follows similarly, based on the corresponding relation
(\ref{eq:pythagoras}) with $\sigma^{B}$ replacing $\sigma^{A}$.\medskip
\end{proof}

\begin{remark}
\emph{Although the related Beppo Levi }$L$\emph{-splines studied in
\cite{L-spl} employ mutually adjoint boundary operators on the extreme
subintervals }$\left(  0,r_{1}\right)  $\emph{ and }$\left(  r_{n}%
,\infty\right)  $\emph{, it can be verified that our non-singular full-space
minimizer }$\sigma^{B}$\emph{ does not share this property. Note that adjoint
boundary conditions were first identified in the context of univariate
interpolation by full-space Mat\'{e}rn kernels \cite{CA11,IMA13} (for a recent
treatment of Mat\'{e}rn kernels on a compact interval, see \cite{CFM14}).}
\end{remark}

\subsection{The limit of Rabut's minimizers on compact intervals}

As stated in the Introduction, for $0<R_{1}<R_{2}<\infty$, Rabut \cite{Rab96} 
used reproducing kernel theory to study the solution $s_{R_{1},R_{2}}$ 
of the problem of minimizing 
${\textstyle\int_{R_{1}}^{R_{2}}}
(\,r\left\vert f^{\prime\prime}\right\vert ^{2}+\tfrac{1}{r}\left\vert
f^{\prime}\right\vert ^{2})\,\mathrm{d}r$ 
in place of (\ref{eq:semin0}),
subject to prescribed values $\nu_{1}$, \ldots, $\nu_{n}$ of $f$ at 
given points $r_{1}$, \ldots, $r_{n}$ of the
compact interval $\left[  R_{1},R_{2}\right]$.

\begin{proposition}
\label{prop}
Let $\nu_{1}$, \ldots, $\nu_{n}$ be arbitrary real values
and $\sigma^{B}$ the non-singular Beppo Levi $L_{0}$-spline interpolant 
to these values at $r_{1}$, \ldots, $r_{n}$, as in Theorem~\ref{thm:EU}. 
Then $\sigma^{B}$ coincides with the pointwise limit, as $R_{1}\rightarrow0$ and
$R_{2}\rightarrow\infty$, of Rabut's interpolant $s_{R_{1},R_{2}}$ to the same data.
\end{proposition}

\begin{proof}
Let $H_{R_{1},R_{2}}$ denote the reproducing kernel \cite[(3.2)]{Rab96}
\begin{eqnarray*}
H_{R_{1},R_{2}} (r,t) & = &  K(r,t) - K(r,r_1) - K(r_1,t) \\
& & \mbox{} + \frac{1}{8(R_2^2 - R_1^2)}
\left( r^2 - r_1^2 + 2 R_1^2 \ln \frac{r}{r_1} \right)
\left( t^2 - r_1^2 + 2 R_2^2 \ln \frac{t}{r_1} \right),
\end{eqnarray*}
where $K$ is defined by \cite[(3.1)]{Rab96}
\begin{equation}
K(r,t) = \frac{1}{4} \left\{
\begin{array}[c]{ll}
0, & \text{if }r\leq t, \\
t^2 - r^2 + ( r^2 + t^2 ) \ln \frac{r}{t}, & \text{if }r>t. 
\end{array}
\right.
\label{eq:rk-truncated}
\end{equation}
Then the minimizer $s_{R_{1},R_{2}}$ can be expressed as \cite[(2.9)]{Rab96}
\begin{equation}
s_{R_{1},R_{2}}(r) = \nu_{1} + \sum_{j=2}^n \lambda_j H_{R_{1},R_{2}}(r,r_j),
\quad R_1 \leq r \leq R_2,
\label{eq:rk-rep}
\end{equation}
for some real coefficients $\lambda_{2}$, \ldots, $\lambda_{n}$. Note that
$ H_{R_{1},R_{2}} (r_1,t) = 0 $ implies $ s_{R_{1},R_{2}} (r_1) = \nu_1 $.
Therefore the coefficients of (\ref{eq:rk-rep}) satisfy the system
\begin{equation}
\sum_{j=2}^n \lambda_j H_{R_{1},R_{2}}(r_k,r_j) = \nu_k - \nu_1 ,
\quad k=2,\ldots,n.
\label{eq:rk-sys}
\end{equation}
Since, by \cite[Theorem 6(ii)]{Rab96}, this system admits a solution 
for arbitrary right-hand side data, it follows that its matrix is invertible. Let
\begin{equation*}
d(R_1,R_2) := \det \left( H_{R_{1},R_{2}}(r_k,r_j) \right)_{k,j=2}^n \not = 0,
\end{equation*}
and, for $j=2,\ldots,n$, denote by $d_j(R_1,R_2)$ the determinant obtained by 
replacing the $j$-th column of the matrix of (\ref{eq:rk-sys}) by the vector of 
right-hand side values of (\ref{eq:rk-sys}). Then, by Cramer's rule, we have
\begin{equation}
\lambda_j = \frac {d_j(R_1,R_2)} {d(R_1,R_2)} ,
\quad j=2,\ldots,n.
\label{eq:cramer}
\end{equation}

Let $ H_{0,\infty} $ the pointwise limit of $ H_{R_{1},R_{2}} $ as 
$R_{1}\rightarrow0$ and $R_{2}\rightarrow\infty$, hence 
\begin{equation}
H_{0,\infty} (r,t) =  K(r,t) - K(r,r_1) - K(r_1,t) + \frac{1}{4} (r^2 - r_1^2)
\ln \frac{t}{r_1} ,
\label{eq:rk-lim}
\end{equation}
and let $d(0,\infty)$ and $d_j(0,\infty)$ the corresponding limit values of 
$d(R_1,R_2)$ and $d_j(R_1,R_2)$. In order to obtain the pointwise limit of 
expression (\ref{eq:rk-rep}), we need to establish the existence of 
limit values for each of the coefficients $\lambda_{2}$, \ldots, $\lambda_{n}$. 
This amounts, via (\ref{eq:cramer}), to showing that $d(0,\infty) \not = 0$, 
\emph{i.e.}\ the limit system obtained by replacing 
$ H_{R_{1},R_{2}} $ by $ H_{0,\infty} $ in (\ref{eq:rk-sys}) has a 
non-singular matrix.

To prove the latter claim, assume that $\nu_k = 0 $, $k=1,\ldots,n$, so the 
above limit system becomes homogeneous, and denote by $c_{2}$, \ldots, 
$c_{n}$, an arbitrary solution of this homogeneous system. Letting 
\begin{equation}
s(r) := \sum_{j=2}^n c_j H_{0,\infty}(r,r_j),
\quad 0 <  r ,
\label{eq:rk-rep-hom}
\end{equation}
it follows that 
\begin{equation*}
s(r_k) = 0, \quad k=2,\ldots,n,
\end{equation*}
while $s(r_1) = 0$ since $ H_{0,\infty}(r_1,t) = 0 $. Also, (\ref{eq:rk-lim}) 
and (\ref{eq:rk-rep-hom}) imply 
\begin{equation}
s(r) = a r^2 + b + \sum_{j=1}^n c_j K(r,r_j),
\quad 0 <  r ,
\label{eq:rk-rep-tru}
\end{equation}
where $c_1 := - \sum_{j=2}^n c_j$, 
$a := \frac{1}{4} \sum_{j=2}^n c_j \ln \frac{r_j}{r_1}$, 
and $b := - \frac{r_1^2}{4} \sum_{j=2}^n c_j \ln \frac{r_j}{r_1}$. Since, 
for each fixed $t$, $K(r,t)$ is a $C^2$-continuous function of $r$, it is 
straightforward to verify that 
$s$ is in fact a non-singular Beppo Levi $L_{0}$-spline in the
sense of our Definition~\ref{def:natLspline}. Invoking 
Theorem~\ref{thm:EU}, part (b), it follows that $s$ vanishes identically. 
Therefore, using (\ref{eq:rk-truncated}) and (\ref{eq:rk-rep-tru}) to 
obtain the form of $s$ successively on each interval $(0,r_1)$, $(r_1,r_2)$, 
\ldots, $(r_n, \infty)$, we deduce $c_j = 0 $, $j=2,\ldots,n$, hence 
$d(0,\infty) \not = 0$, as claimed.

Letting $\widehat{\lambda}_j := d_j(0,\infty) / d(0,\infty)$, $j=2,\ldots,n$, 
and defining
\begin{equation*}
s_{0,\infty}(r) := \nu_{1} + \sum_{j=2}^n 
\widehat{\lambda}_j H_{0,\infty}(r,r_j),
\quad 0 < r ,
\end{equation*}
it follows that $s_{0,\infty}$ is the pointwise limit of $ s_{R_{1},R_{2}}$.
Further, $s_{0,\infty} (r_k) = \nu_k$, $k=1,\ldots,n$, and  $s_{0,\infty}$ 
admits a representation similar to (\ref{eq:rk-rep-tru}), hence 
$s_{0,\infty} \in \mathcal{S}_0 (\rho)$. Therefore 
Theorem~\ref{thm:EU}, part (b), implies $s_{0,\infty}\equiv\sigma^{B}$.
\end{proof}

\begin{remark}
\emph{Unlike }$\sigma^{B}$\emph{, Rabut's minimizer }$s_{R_{1},R_{2}}$\emph{
satisfies the `natural' boundary conditions from classical calculus of
variations. Indeed, expression (\ref{eq:rk-rep}) and the last two lines of
\cite[(3.4)]{Rab96} show that the end conditions for }$s_{R_{1},R_{2}}$\emph{
are formulated by means of identical differential operators at both endpoints
}$R_{1}$\emph{ and }$R_{2}$\emph{, in agreement with the `natural'/`type II'
conditions of }$L$\emph{-spline literature \cite{L70} or \cite[Theorem
4]{JP72}.}
\end{remark}

\subsection{Linear representation with dilates of a basis function}

We now obtain a representation of the two types of interpolatory Beppo Levi
$L_{0}$-splines in terms of the following basis function:%
\begin{equation}
\varphi_{0}\left(  r\right)  :=\left\{
\begin{array}
[c]{l}%
r^{2}-r^{2}\ln r,\quad0\leq r\leq1,\\
1+\ln r,\quad1<r.
\end{array}
\right.  \label{eq:rbf-ann0}%
\end{equation}
Note that $\varphi_{0}$ is $C^{2}$-continuous at $r=1$, $G_{0}\varphi
_{0}\left(  r\right)  =0$ for $0<r<1$, $R_{0}\varphi_{0}\left(  r\right)  =0$
for $1<r$, and $\varphi_{0}\left(  0\right)  =0$.

\begin{figure}
\centering
\includegraphics[height=2.0in,width=3.0in]{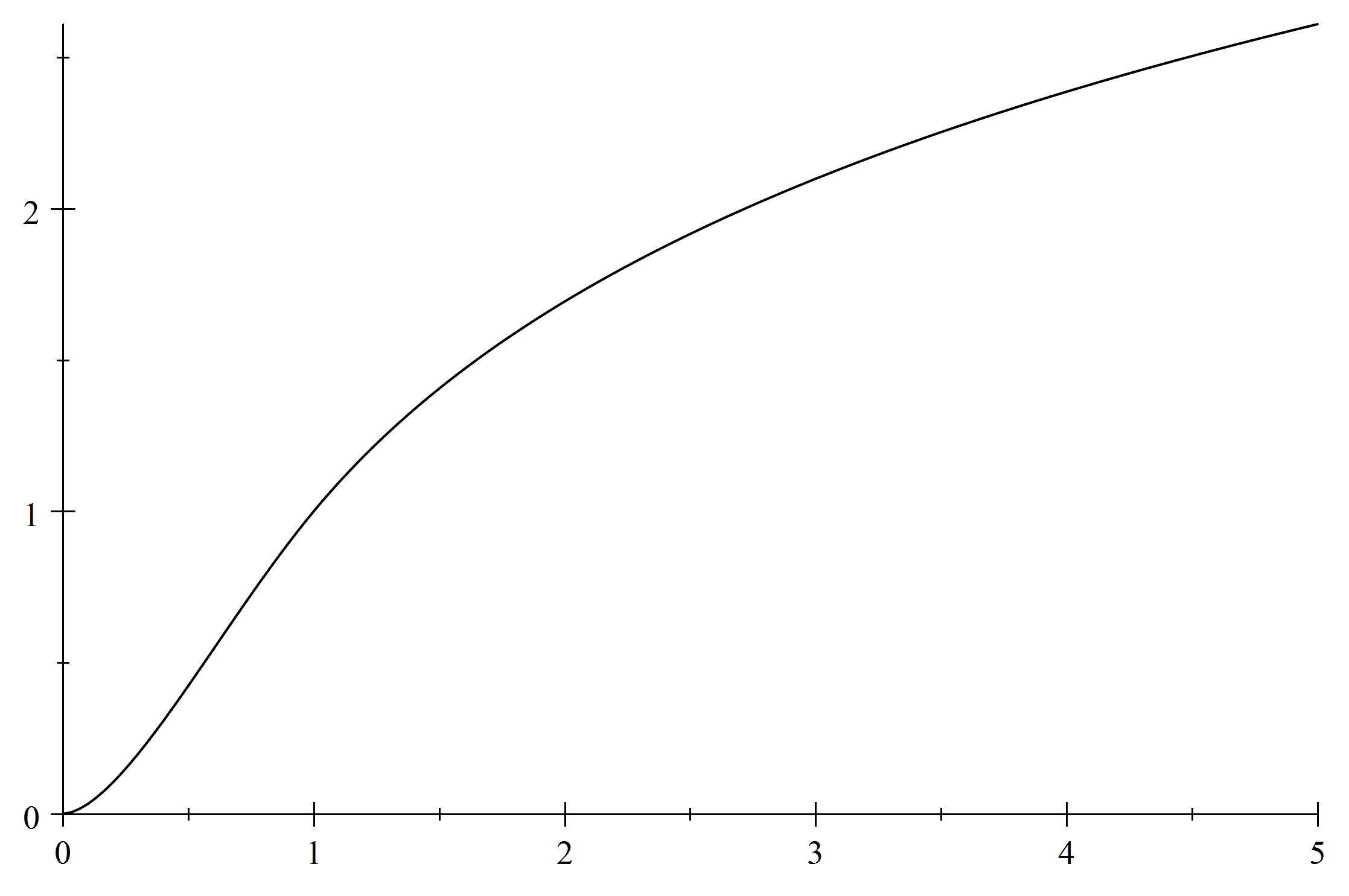}
\caption{Plot of the basis function $\varphi_{0}$.}
\label{phinot}
\end{figure}

The expression stated in the next theorem will be used in section~\ref{sec5}
to evaluate our Beppo Levi $L_{0}$-spline interpolants on a fine mesh and
estimate the numerical accuracy of this approximation procedure. Similar
representations for the related family of $L$-splines \cite[Lemma 3]{MMA12} have
played a crucial role in the construction of Beppo Levi polyspline surfaces
which interpolate smooth curves prescribed on concentric circles.

\begin{theorem}
Let $\sigma$ be one of the Beppo Levi $L_{0}$-splines $\sigma^{A}$ or
$\sigma^{B}$ satisfying the interpolation conditions \emph{(\ref{eq:typeA})}
or \emph{(\ref{eq:typeB})} of Theorem~\ref{thm:EU} for given values $\alpha$,
$\nu_{1}$, \ldots, $\nu_{n}$. There exist unique coefficients $c$, $a_{1}$,
\ldots, $a_{n}$, such that, for all $r\geq0$, we have the representation%
\begin{equation}
\sigma\left(  r\right)  =c+\sum\limits_{k=1}^{n}a_{k}\varphi_{0}\left(
\frac{r}{r_{k}}\right)  , \label{eq:rbf-k0}%
\end{equation}
where, if $\sigma=\sigma^{B}$, the coefficients also satisfy:%
\begin{equation}
\sum_{k=1}^{n}a_{k}r_{k}^{-2}=0. \label{eq:side-B}%
\end{equation}

\end{theorem}

\begin{proof}
Letting $\sigma_{1}\left(  r\right)  $ denote the right-hand side of
(\ref{eq:rbf-k0}), the noted properties of $\varphi_{0}$ imply $\sigma_{1}%
\in\mathcal{S}_{0}\left(  \rho\right)  $ and $c=\sigma_{1}\left(  0\right)  $.
Also, $\sigma_{1}\ $is non-singular if and only if (\ref{eq:side-B}) holds.
Therefore, invoking Theorem~\ref{thm:EU}, we have to prove the existence of
unique coefficients such that either $\sigma_{1}$ satisfies (\ref{eq:typeA})
in place of $\sigma^{A}$, or (\ref{eq:typeB}) and (\ref{eq:side-B}) hold with
$\sigma_{1}$ in place of $\sigma^{B}$. Since these equations are linear, it is
sufficient to assume that all data values $\alpha$, $\nu_{1}$, \ldots,
$\nu_{n}$ are zero and prove that each of the two resulting homogeneous
systems admits only the trivial solution.

In each of the two cases,
from $\sigma_{1}\in\mathcal{S}_{0}\left(  \rho\right)$
and Theorem~\ref{thm:EU}, part (b), it follows that $\sigma_{1}$ 
is identically zero on each subinterval. Since
the form of $\sigma_{1}$ on the subinterval $(0,r_1)$ is
\[ 
\sigma_{1}\left(  r\right)  = c + r^{2}\sum\limits_{k=1}^{n} 
\frac{ a_{k}}{r_{k}^{2}} (1+\ln r_{k}) 
-  r^{2}\ln r \sum\limits_{k=1}^{n} \frac{ a_{k}}{r_{k}^{2}},
\quad r\in\left(  0,r_{1}\right) ,
\]
we obtain, in particular, $c=0$. Also, since $a_1$ is the coefficient of $\ln r$
in the expression of $\sigma_{1}$ in terms of the basis (\ref{eq:ker}) 
on the subinterval $(r_1,r_2)$, we obtain $a_1 = 0$, and applying a similar 
argument on successive subintervals it follows that $a_2 = \ldots = a_n = 0$. 
This establishes the theorem.
\end{proof}

\begin{remark}
\emph{Although each dilation }$\cdot/r_{k}$\emph{ can be regarded as a
`translation' in the multiplicative topological group }$\left(  0,\infty
\right)  $\emph{, our representation (\ref{eq:rbf-k0}) does not belong to the
framework of interpolation via Hankel translates recently considered by
Arteaga \& Marrero \cite{AM12}.}
\end{remark}

\section{Convergence orders}

In this section, we establish $L^{p}$-error bounds for the interpolatory
profiles studied in the previous section. The method of our proofs is based on
the error analysis \cite{ANW67,SV67} developed for generalized splines and $L$-splines.

\begin{theorem}
\label{thm:Linfty}For $n\geq2$ and $r_{0}=0<r_{1}<\ldots<r_{n}<\infty=r_{n+1}%
$, let $\rho:=\left\{  r_{1},\ldots,r_{n}\right\}  $ and $h:=\max_{1\leq j\leq
n-1}\left(  r_{j+1}-r_{j}\right)  $. Given $f\in\Lambda_{0}$, let
$\alpha:=f\left(  0\right)  $ and $\nu_{j}:=f\left(  r_{j}\right)  $ for
$1\leq j\leq n$. If $\sigma$ denotes one of the corresponding Beppo Levi
$L_{0}$-spline interpolants $\sigma^{A}$ or $\sigma^{B}$ obtained in
Theorem~\ref{thm:EU}, then, for $l\in\left\{  0,1\right\}  $,%
\begin{equation}
\left\Vert f^{\left(  l\right)  }-\sigma^{\left(  l\right)  }\right\Vert
_{L^{\infty}\left[  r_{1},r_{n}\right]  }\leq\frac{1}{2^{1-l}\sqrt{r_{1}}%
}h^{3/2-l}\left\Vert f\right\Vert _{0}. \label{eq:Linfty-err}%
\end{equation}

\end{theorem}

\begin{proof}
Let $\psi:=f-\sigma$, so $\psi\left(  r_{j}\right)  =0$ for $1\leq j\leq n$.
Since $\psi\in C^{1}\left(  0,\infty\right)  $, Rolle's theorem implies that,
for each $j\in\left\{  1,\ldots,n-1\right\}  $, there exists $t_{j}\in\left(
r_{j},r_{j+1}\right)  $, such that%
\[
\psi^{\prime}\left(  t_{j}\right)  =0.
\]
Let $r^{\ast},t^{\ast}\in\left[  r_{1},r_{n}\right]  $ satisfy%
\begin{equation}
\left\vert \psi\left(  r^{\ast}\right)  \right\vert =\left\Vert \psi
\right\Vert _{L^{\infty}\left[  r_{1},r_{n}\right]  },\quad\left\vert
\psi^{\prime}\left(  t^{\ast}\right)  \right\vert =\left\Vert \psi^{\prime
}\right\Vert _{L^{\infty}\left[  r_{1},r_{n}\right]  }, \label{eq:abs-max}%
\end{equation}
and choose $k\in\left\{  1,\ldots,n\right\}  $ and $m\in\left\{
1,\ldots,n-1\right\}  $, such that%
\begin{equation}
\left\vert r^{\ast}-r_{k}\right\vert \leq h/2,\quad\left\vert t^{\ast}%
-t_{m}\right\vert \leq h. \label{eq:closest-knot}%
\end{equation}
Since $\psi^{\prime}$ is locally absolutely continuous, it follows that
$\psi^{\prime\prime}$ is locally integrable and we have the Leibniz-Newton
formulae:%
\begin{equation}%
\begin{array}
[c]{l}%
\psi\left(  r^{\ast}\right)  =\psi\left(  r^{\ast}\right)  -\psi\left(
r_{k}\right)  =\int_{r_{k}}^{r^{\ast}}\psi^{\prime}\left(  r\right)
\mathrm{d}r,\\
\psi^{\prime}\left(  t^{\ast}\right)  =\psi^{\prime}\left(  t^{\ast}\right)
-\psi^{\prime}\left(  t_{m}\right)  =\int_{t_{m}}^{t^{\ast}}\psi^{\prime
\prime}\left(  r\right)  \mathrm{d}r.
\end{array}
\label{eq:LeibNewt}%
\end{equation}
Therefore the first line of (\ref{eq:LeibNewt}) and (\ref{eq:closest-knot})
imply the estimate%
\begin{equation}
\left\vert \psi\left(  r^{\ast}\right)  \right\vert \leq\left\vert \int
_{r_{k}}^{r^{\ast}}\left\vert \psi^{\prime}\left(  r\right)  \right\vert
\mathrm{d}r\right\vert \leq\frac{h}{2}\left\Vert \psi^{\prime}\right\Vert
_{L^{\infty}\left[  r_{1},r_{n}\right]  }, \label{eq:bound1}%
\end{equation}
while, using Cauchy-Schwarz and (\ref{eq:closest-knot}) in the second line of
(\ref{eq:LeibNewt}), we obtain%
\begin{align}
\left\vert \psi^{\prime}\left(  t^{\ast}\right)  \right\vert  &
\leq\left\vert \int_{t_{m}}^{t^{\ast}}\left\vert \psi^{\prime\prime}\left(
r\right)  \right\vert \mathrm{d}r\right\vert =\left\vert \int_{t_{m}}%
^{t^{\ast}}r^{-1/2}\left[  r^{1/2}\left\vert \psi^{\prime\prime}\left(
r\right)  \right\vert \right]  \mathrm{d}r\right\vert \nonumber\\
&  \leq\left\vert \int_{t_{m}}^{t^{\ast}}r^{-1}\mathrm{d}r\right\vert
^{1/2}\left\vert \int_{t_{m}}^{t^{\ast}}r\,\left\vert \psi^{\prime\prime
}\left(  r\right)  \right\vert ^{2}\mathrm{d}r\right\vert ^{1/2}\nonumber\\
&  \leq\left(  \frac{\left\vert t^{\ast}-t_{m}\right\vert }{r_{1}}\right)
^{1/2}\left(  \int_{0}^{\infty}r\,\left\vert \psi^{\prime\prime}\left(
r\right)  \right\vert ^{2}\mathrm{d}r\right)  ^{1/2}\nonumber\\
&  \leq r_{1}^{-1/2}h^{1/2}\left\Vert \psi\right\Vert _{0}. \label{eq:bound2}%
\end{align}

On the other hand, the first integral relation (\ref{eq:pythagoras}), valid
for both $\sigma^{A}$ and $\sigma^{B}$, implies%
\begin{equation}
\left\Vert \psi\right\Vert _{0}=\left\Vert f-\sigma\right\Vert _{0}%
\leq\left\Vert f\right\Vert _{0}. \label{eq:bound3}%
\end{equation}
The conclusion follows from (\ref{eq:abs-max}) and (\ref{eq:bound1}%
)--(\ref{eq:bound3}).
\end{proof}

\begin{theorem}
\label{thm:L2}Under the hypotheses of Theorem~\ref{thm:Linfty}, for
$l\in\left\{  0,1\right\}  $, we have%
\begin{equation}
\left\Vert f^{\left(  l\right)  }-\sigma^{\left(  l\right)  }\right\Vert
_{L^{2}\left[  r_{1},r_{n}\right]  }\leq\frac{1}{2^{1-l}\sqrt{r_{1}}}%
h^{2-l}\left\Vert f\right\Vert _{0}. \label{eq:L2-error}%
\end{equation}

\end{theorem}

\begin{proof}
We employ the notations from the proof of Theorem~\ref{thm:Linfty}, in
particular $\psi:=f-\sigma$. For $j\in\left\{  1,\ldots,n-1\right\}  $, we use
the following univariate version of the well-known Friedrichs inequality:%
\begin{equation}
\int_{r_{j}}^{r_{j+1}}\left\vert \psi\left(  r\right)  \right\vert
^{2}\mathrm{d}r\leq\frac{h^{2}}{4}\int_{r_{j}}^{r_{j+1}}\left\vert
\psi^{\prime}\left(  r\right)  \right\vert ^{2}\mathrm{d}r. \label{eq:Frid}%
\end{equation}
Indeed, if $m_{j}:=\frac{r_{j}+r_{j+1}}{2}$ and $r_{j}\leq r\leq m_{j}$, then
Leibniz-Newton formula, $\psi\left(  r_{j}\right)  =0$, and Cauchy-Schwarz
imply%
\[
\left\vert \psi\left(  r\right)  \right\vert ^{2}=\left\vert \int_{r_{j}}%
^{r}\psi^{\prime}\left(  u\right)  \mathrm{d}u\right\vert ^{2}\leq\left(
r-r_{j}\right)  \int_{r_{j}}^{r}\left\vert \psi^{\prime}\left(  u\right)
\right\vert ^{2}\mathrm{d}u\leq\frac{h}{2}\int_{r_{j}}^{m_{j}}\left\vert
\psi^{\prime}\left(  u\right)  \right\vert ^{2}\mathrm{d}u,
\]
hence, by integration,%
\[
\int_{r_{j}}^{m_{j}}\left\vert \psi\left(  r\right)  \right\vert
^{2}\mathrm{d}r\leq\frac{h^{2}}{4}\int_{r_{j}}^{m_{j}}\left\vert \psi^{\prime
}\left(  r\right)  \right\vert ^{2}\mathrm{d}r.
\]
Now (\ref{eq:Frid}) is obtained by adding this to a similar inequality that
holds on the interval $\left[  m_{j},r_{j+1}\right]  $. 

Summing both sides of
(\ref{eq:Frid}) over $j$ and taking square roots, we obtain%
\begin{equation}
\left\Vert \psi\right\Vert _{L^{2}\left[  r_{1},r_{n}\right]  }\leq\frac{h}%
{2}\left\Vert \psi^{\prime}\right\Vert _{L^{2}\left[  r_{1},r_{n}\right]  }.
\label{eq:bound5}%
\end{equation}
Next, note that $\psi^{\prime\prime}$ is square integrable over $\left[
r_{1},r_{n}\right]  $, due to%
\begin{align}
\left\Vert \psi^{\prime\prime}\right\Vert _{L^{2}\left[  r_{1},r_{n}\right]
}^{2}  &  =\int_{r_{1}}^{r_{n}}\frac{1}{r}\,r\,\left\vert  \psi^{\prime\prime
}\left(  r\right)  \right\vert  ^{2}\mathrm{d}r\nonumber\\
&  \leq\frac{1}{r_{1}}\int_{r_{1}}^{r_{n}}r\,\left\vert  \psi^{\prime\prime
}\left(  r\right)  \right\vert  ^{2}\mathrm{d}r\nonumber\\
&  \leq\frac{1}{r_{1}}\left\Vert \psi\right\Vert _{0}^{2}\leq\frac{1}{r_{1}%
}\left\Vert f\right\Vert _{0}^{2}. \label{eq:bound4}%
\end{align}
Hence, for $j\in\left\{  1,\ldots,n-1\right\}  $ and $r_{j}\leq r\leq r_{j+1}%
$, since $\psi^{\prime}$ is locally absolutely continuous and $\psi^{\prime
}\left(  t_{j}\right)  =0$, where $r_{j}<t_{j}<r_{j+1}$, we have
\[
\left\vert \psi^{\prime}\left(  r\right)  \right\vert ^{2}=\left\vert
\int_{t_{j}}^{r}\psi^{\prime\prime}\left(  u\right)  \mathrm{d}u\right\vert
^{2}\leq\left\vert r-t_{j}\right\vert \int_{t_{j}}^{r}\left\vert \psi
^{\prime\prime}\left(  u\right)  \right\vert ^{2}\mathrm{d}u\leq h\int_{r_{j}%
}^{r_{j+1}}\left\vert \psi^{\prime\prime}\left(  u\right)  \right\vert
^{2}\mathrm{d}u,
\]
which implies, via integration, the following analog of (\ref{eq:Frid}):%
\[
\int_{r_{j}}^{r_{j+1}}\left\vert  \psi^{\prime}\left(  r\right)  \right\vert
^{2}\mathrm{d}r\leq h^{2}\int_{r_{j}}^{r_{j+1}}\left\vert  \psi^{\prime\prime
}\left(  r\right)  \right\vert  ^{2}\mathrm{d}r.
\]
Summing once more over $j$ and using (\ref{eq:bound4}), we obtain%
\[
\left\Vert \psi^{\prime}\right\Vert _{L^{2}\left[  r_{1},r_{n}\right]  }%
^{2}\leq h^{2}\left\Vert \psi^{\prime\prime}\right\Vert _{L^{2}\left[
r_{1},r_{n}\right]  }^{2}\leq\frac{1}{r_{1}}h^{2}\left\Vert f\right\Vert
_{0}^{2},
\]
which establishes (\ref{eq:L2-error}) via (\ref{eq:bound5}).
\end{proof}

\begin{remark}
\emph{The result of Theorem~\ref{thm:L2} is used in \cite{L-spl}, along with
similar convergence results for the related class of }$L$\emph{-splines
studied there, to establish an }$L^{2}$\emph{-error bound for transfinite
surface interpolation with biharmonic Beppo Levi polysplines on annuli.}
\end{remark}

\begin{corollary}
\label{thm:Lp}Under the hypotheses of Theorem~\ref{thm:Linfty}, if
$p\in\left[  2,\infty\right]  $ and $l\in\left\{  0,1\right\}  $, then%
\begin{equation}
\left\Vert f^{\left(  l\right)  }-\sigma^{\left(  l\right)  }\right\Vert
_{L^{p}\left[  r_{1},r_{n}\right]  }\leq\frac{1}{2^{1-l}\sqrt{r_{1}}%
}h^{3/2+1/p-l}\left\Vert f\right\Vert _{0}. \label{eq:Lp-error}%
\end{equation}

\end{corollary}

\begin{proof}
We employ a classical result on `interpolation between $L^{p}$-spaces' (see
\cite[p. 175]{BS88}): if $1\leq p_{0}<p<p_{1}\leq\infty$ and $\frac{1}%
{p}=\frac{1-\theta}{p_{0}}+\frac{\theta}{p_{1}}$ for some $\theta\in\left(
0,1\right)  $, then%
\[
\left\Vert \psi\right\Vert _{L^{p}}\leq\left\Vert \psi\right\Vert _{L^{p_{0}}%
}^{1-\theta}\left\Vert \psi\right\Vert _{L^{p_{1}}}^{\theta},\quad\forall
\psi\in L^{p_{0}}\cap L^{p_1}.
\]
Letting $p_{0}=2$, $p_{1}=\infty$, it follows that $2\left(  1-\theta\right)
+\frac{3}{2}\theta=\frac{3}{2}+\frac{1}{p}$, hence (\ref{eq:Lp-error}) is a
consequence of (\ref{eq:Linfty-err}) and (\ref{eq:L2-error}). Alternatively, a
direct proof of (\ref{eq:Lp-error}) for $p\in\left(  2,\infty\right)  $ can be
obtained via Jensen's inequality, as in \cite[Chapter 6]{LLS07}.
\end{proof}

\begin{remark}
\emph{For }$1\leq p<2$\emph{, the usual embedding of }$L^{2}\left[
r_{1},r_{n}\right]  $\emph{ into }$L^{p}\left[  r_{1},r_{n}\right]  $\emph{
shows that the }$L^{2}$\emph{-approximation order }$O\left(  h^{2-l}\right)
$\emph{ of (\ref{eq:L2-error}) also applies to the corresponding }$L^{p}%
$\emph{-norm of the error.}
\end{remark}

The final theorem of this section shows that the exponent $3/2$ appearing in
Theorem~\ref{thm:Linfty} is asymptotically sharp as $h\rightarrow0$, in the
sense that it cannot be increased for the class $\Lambda_{0}$ of data
functions. A similar result holds for the approximation order obtained in
Theorem~\ref{thm:L2}, but is omitted, for brevity. We will make use of the 
following lemma.

\begin{lemma}
\label{le:best-app}Let $\left(  X,\left\Vert \cdot\right\Vert \right)  $ be a
normed vector space of dimension at least $d+1$ and denote by $\mathcal{H}%
\subset X^{d}$ the set of all $d$-tuples $\left(  \mathbf{v}_{1}%
,\ldots,\mathbf{v}_{d}\right)  $ of $d$ linearly independent vectors in $X$.
For a given vector $\mathbf{x}\in X$ and any $\left(  \mathbf{v}_{1}%
,\ldots,\mathbf{v}_{d}\right)  \in\mathcal{H}$, let $\delta_{\mathbf{x}%
}\left(  \mathbf{v}_{1},\ldots,\mathbf{v}_{d}\right)  $ be the distance from
$\mathbf{x}$ to the sub-space of all linear combinations of $\mathbf{v}%
_{1},\ldots,\mathbf{v}_{d}$, i.e.%
\[
\delta_{\mathbf{x}}\left(  \mathbf{v}_{1},\ldots,\mathbf{v}_{d}\right)
=\inf_{\mathbf{v\in}\,\mathrm{span}\left\{  \mathbf{v}_{1},\ldots
,\mathbf{v}_{d}\right\}  }\left\Vert \mathbf{x-v}\right\Vert .
\]
Then $\delta_{\mathbf{x}}$ is a continuous function of $\left(  \mathbf{v}%
_{1},\ldots,\mathbf{v}_{d}\right)  $ on $\mathcal{H}$, where $\mathcal{H}$ is
endowed with the product topology induced from $X^{d}$.
\end{lemma}

\begin{proof}
Let $\left(  \mathbf{v}_{1},\ldots,\mathbf{v}_{d}\right)  \in\mathcal{H}$,
$\delta:=\delta_{\mathbf{x}}\left(  \mathbf{v}_{1},\ldots,\mathbf{v}%
_{d}\right)  $, and, for each $j\in\left\{  1,\ldots,d\right\}  $, consider a
sequence $\left\{  \mathbf{v}_{j,n}\right\}  _{n=1}^{\infty}$ of vectors in
$X$, such that $\left\Vert \mathbf{v}_{j,n}-\mathbf{v}_{j}\right\Vert
\rightarrow0$ as $n\rightarrow\infty$, and $\left(  \mathbf{v}_{1,n}%
,\ldots,\mathbf{v}_{d,n}\right)  \in\mathcal{H}$, $\forall n\geq1$. We have to
show that, for every $\epsilon>0$, the following inequalities hold for all $n$
sufficiently large :%
\begin{equation}
\delta-\epsilon\leq\delta_{\mathbf{x}}\left(  \mathbf{v}_{1,n},\ldots
,\mathbf{v}_{d,n}\right)  \leq\delta+\epsilon. \label{eq:app1}%
\end{equation}

Letting $a_{1},\ldots,a_{d}$ be scalars for which $\left\Vert \mathbf{x}%
-\sum_{j=1}^{d}a_{j}\mathbf{v}_{j}\right\Vert =\delta$, and setting
$\mathbf{v}:=\sum_{j=1}^{d}a_{j}\mathbf{v}_{j,n}$, we have%
\[
\left\Vert \mathbf{x}-\mathbf{v}\right\Vert \leq\left\Vert \mathbf{x}%
-\sum_{j=1}^{d}a_{j}\mathbf{v}_{j}\right\Vert +\sum_{j=1}^{d}\left\vert
a_{j}\right\vert \left\Vert  \mathbf{v}_{j}-\mathbf{v}_{j,n} \right\Vert .
\]
Hence $\left\Vert \mathbf{x}-\mathbf{v}\right\Vert \leq\delta+\epsilon$ for
$n$ sufficiently large, which implies the right-side inequality of
(\ref{eq:app1}).

For the remaining inequality, we argue by contradiction and assume that there
exists $\epsilon>0$ such that, after selecting and re-indexing a sub-sequence,
we have $\delta_{\mathbf{x}}\left(  \mathbf{v}_{1,n},\ldots,\mathbf{v}%
_{d,n}\right)  <\delta-\epsilon$, for all $n\geq1$. Hence, there exist scalar
sequences $\left\{  b_{j,n}\right\}  _{n=1}^{\infty}$, $j\in\left\{
1,\ldots,d\right\}  $, such that, if $\mathbf{w}_{n}:=\sum_{j=1}^{d}%
b_{j,n}\mathbf{v}_{j,n}$, then%
\begin{equation}
\left\Vert \mathbf{x}-\mathbf{w}_{n}\right\Vert <\delta-\epsilon,\quad\forall
n\geq1. \label{eq:app2}%
\end{equation}
Next, we invoke a well-known result \cite[Lemma 2.4-1]{Krzg}, which guarantees
the existence of a constant $K>0$, such that%
\[
\left\Vert c_{1}\mathbf{v}_{1}+\ldots+c_{d}\mathbf{v}_{d}\right\Vert \geq K
\]
for all scalars $c_{1}\mathbf{,}\ldots,c_{d}$ with $\sum_{j=1}^{d}\left\vert
c_{j}\right\vert =1$. Also, let $n_{0}$ be a positive integer for which
$\left\Vert  \mathbf{v}_{j}-\mathbf{v}_{j,n}  \right\Vert
\leq\frac{K}{2}$, $\forall n\geq n_{0}$, $\forall j\in\left\{  1,\ldots
,d\right\}  $. Then, for all scalars $c_{1}\mathbf{,}\ldots,c_{d}$ as above
and all $n\geq n_{0}$, we have%
\[
\left\Vert \sum_{j=1}^{d}c_{j}\mathbf{v}_{j}\right\Vert -\left\Vert \sum
_{j=1}^{d}c_{j}\mathbf{v}_{j,n}\right\Vert \leq\sum_{j=1}^{d}\left\vert
c_{j}\right\vert \left\Vert  \mathbf{v}_{j}-\mathbf{v}_{j,n}
\right\Vert \leq\frac{K}{2},
\]
hence%
\[
\left\Vert \sum_{j=1}^{d}c_{j}\mathbf{v}_{j,n}\right\Vert \geq\left\Vert
\sum_{j=1}^{d}c_{j}\mathbf{v}_{j}\right\Vert -\frac{K}{2}\geq\frac{K}{2}.
\]
For all $n\geq n_{0}$, we obtain, via (\ref{eq:app2}),%
\[
\frac{K}{2}\sum_{j=1}^{d}\left\vert b_{j,n}\right\vert \leq\left\Vert
\sum_{j=1}^{d}b_{j,n}\mathbf{v}_{j,n}\right\Vert =\left\Vert \mathbf{w}%
_{n}\right\Vert \leq\left\Vert \mathbf{x}\right\Vert +\delta-\epsilon,
\]
and therefore each sequence $\left\{  b_{j,n}\right\}  _{n=1}^{\infty}$,
$j\in\left\{  1,\ldots,d\right\}  $, is bounded. Selecting convergent
sub-sequences, re-indexed such that $\lim_{n\rightarrow\infty}b_{j,n}=:b_{j}$,
and passing to the limit in (\ref{eq:app2}), we deduce the existence of a
vector $\mathbf{w}:=\sum_{j=1}^{d}b_{j}\mathbf{v}_{j}$, such that $\left\Vert
\mathbf{x}-\mathbf{w}\right\Vert \leq\delta-\epsilon$, which contradicts the
definition of $\delta$. The lemma is proved.
\end{proof}

\begin{theorem}
\label{thm:sharpness}For each $n\geq2$, let $h=1/\left(  n-1\right)  $ and let
$\rho_{h}$ be the set of uniformly spaced knots $r_{j}=1+\left(  j-1\right)
h$, $j\in\left\{  1,\ldots,n\right\}  $, hence $\left[  r_{1},r_{n}\right]
=\left[  1,2\right]  $. Then, for any $\varepsilon>0$ arbitrarily small, there
exists $f_{\varepsilon}\in\Lambda_{0}$ and a constant $M>0$, such that%
\[
M\,h^{3/2+\varepsilon}\leq\left\Vert f_{\varepsilon}-\sigma_{h}\right\Vert
_{L^{\infty}\left[  1,2\right]  }%
\]
holds for all sufficiently small $h>0$ and all $\sigma_{h}\in\mathcal{S}%
_{0}\left(  \rho_{h}\right)  $.
\end{theorem}

\begin{proof}
We adapt the arguments of \cite[Theorem~11]{SV67}, using the fact that the
null space $\mathrm{Ker}\,L_{0}$ described by (\ref{eq:ker}) is a
finite-dimensional vector space. Specifically, let $\mu>0$ and, for each
$h\in\left[  0,1\right]  $, define%
\[
\alpha\left(  h\right)  =\alpha\left(  h,\mu\right)  =\inf_{\eta
\in\mathrm{Ker}\,L_{0}}\left\Vert t^{\mu}-\eta\left(  1+th\right)  \right\Vert
_{L^{\infty}\left[  0,1\right]  }.
\]
We claim that $\alpha$ is a continuous function of $h$ on $\left[  0,1\right]
$. To see this, first observe that, for any $\eta\in\mathrm{Ker}\,L_{0}$ and
$h_{0}\in\left[  0,1\right]  $, we have%
\[
\left\Vert \eta\left(  1+\cdot h\right)  -\eta\left(  1+\cdot h_{0}\right)
\right\Vert _{L^{\infty}\left[  0,1\right]  }\rightarrow0,\text{ as
}h\rightarrow h_{0},
\]
which is a consequence of the uniform continuity of any such function $\eta$
on $\left[  1,2\right]  $. Then use this observation for each of the four
basis functions of (\ref{eq:ker}) and apply Lemma~\ref{le:best-app} with
$X:=C\left[  0,1\right]  $, $d:=4$, to deduce $\alpha\left(  h\right)
\rightarrow\alpha\left(  h_{0}\right)  $ as $h\rightarrow h_{0}$.

Also, note that $\alpha\left(  h\right)  >0$, $\forall h\in\left[  0,1\right]
$. Indeed, this inequality is easily verified if $h=0$, while, if $h\in(0,1]$,
it follows from the fact that $\left(  r-1\right)  ^{\mu}\not \in
\mathrm{Ker}\,L_{0}$, for $r\in\left[  1,1+h\right]  $. Therefore, letting
$\min\limits_{0\leq h\leq1}\alpha\left(  h,\mu\right)  =:A\left(  \mu\right)
$, we have $A\left(  \mu\right)  >0$ and%
\begin{align}
\inf_{\eta\in\mathrm{Ker}\,L_{0}}\left\Vert \left(  th\right)  ^{\mu}%
-\eta\left(  1+th\right)  \right\Vert _{L^{\infty}\left[  0,1\right]  }  &
=h^{\mu}\inf_{\eta\in\mathrm{Ker}\,L_{0}}\left\Vert t^{\mu}-h^{-\mu}%
\eta\left(  1+th\right)  
\right\Vert _{L^{\infty}\left[  0,1\right]}\nonumber\\
&  =h^{\mu}\inf_{\zeta\in\mathrm{Ker}\,L_{0}}
\left\Vert t^{\mu}-\zeta\left(1+th\right)  
\right\Vert _{L^{\infty}\left[  0,1\right]  }\nonumber\\
&  \geq h^{\mu}A\left(  \mu\right)  .\label{eq:cheb}
\end{align}

Next, let $g$ be a $C^{\infty}$-smooth function with compact support within
$\left(  0,\infty\right)  $, such that $g\left(  r\right)  =1$, $\forall
r\in\left[  1,2\right]  $, and, for $\varepsilon>0$, define%
\begin{equation}
f_{\varepsilon}\left(  r\right)  :=g\left(  r\right)  \left\vert
r-1\right\vert ^{\frac{3}{2}+\varepsilon},\quad r\in\left(  0,\infty\right)  .
\label{eq:f-eps}%
\end{equation}
It follows that $f_{\varepsilon}\in\Lambda_{0}$ for each
$\varepsilon>0$. Making the change of variables $r-1=th$, $t\in\left[
0,1\right]  $, and using (\ref{eq:cheb}) with $\mu:=\tfrac{3}{2}+\varepsilon$,
we obtain, for any $\sigma_{h}\in\mathcal{S}_{0}\left(  \rho_{h}\right)  $,%
\begin{align*}
\left\Vert f_{\varepsilon}-\sigma_{h}\right\Vert _{L^{\infty}\left[
1,2\right]  }  &  \geq\left\Vert f_{\varepsilon}-\sigma_{h}\right\Vert
_{L^{\infty}\left[  1,1+h\right]  }\geq\inf_{\eta\in\mathrm{Ker}\,L_{0}%
}\left\Vert f_{\varepsilon}-\eta\right\Vert _{L^{\infty}\left[  1,1+h\right]
}\\
&  =\inf_{\eta\in\mathrm{Ker}\,L_{0}}\left\Vert \left(  r-1\right)  ^{\frac
{3}{2}+\varepsilon}-\eta\left(  r\right)  \right\Vert _{L^{\infty}\left[
1,1+h\right]  }\\
&  =\inf_{\eta\in\mathrm{Ker}\,L_{0}}\left\Vert \left(  th\right)  ^{\frac
{3}{2}+\varepsilon}-\eta\left(  1+th\right)  \right\Vert _{L^{\infty}\left[
0,1\right]  }\geq h^{\frac{3}{2}+\varepsilon}A\left(  \tfrac{3}{2}%
+\varepsilon\right)  ,
\end{align*}
hence the conclusion of the theorem holds with $M:=A\left(  \tfrac{3}%
{2}+\varepsilon\right)  $.
\end{proof}

\section{Numerical results and examples}

\label{sec5}

\subsection{Numerical accuracy of $L_{0}$-spline interpolation}

For a data profile $f\in\Lambda_{0}$, we denote by $\sigma^{A}$ and
$\sigma^{B}$ the Beppo Levi $L_{0}$-spline interpolants described in
Theorem~\ref{thm:Linfty}. To test the numerical accuracy of these two
profiles, we partition the interval $\left[  1,2\right]  $ into $n-1$ equal
subintervals of mesh-size $h:=\frac{1}{n-1}$ by letting $r_{j}:=1+h\left(
j-1\right)  $, $j\in\left\{  1,\ldots,n\right\}  $, hence $r_{1}=1$ and
$r_{n}=2$.

Given $n$ and $f$, we use Matlab to compute the coefficients of representation
(\ref{eq:rbf-k0}) for $\sigma^{A}$ and $\sigma^{B}$. Specifically, if
$\sigma:=\sigma^{A}$, we let $c:=f\left(  0\right)  $ and find $\left\{
a_{k}\right\}  _{k=1}^{n}$ from the system%
\[
\sum\limits_{k=1}^{n}a_{k}\varphi_{0}\left(  \frac{r_{j}}{r_{k}}\right)
=f\left(  r_{j}\right)  -f\left(  0\right)  ,\quad\forall j\in\left\{
1,\ldots,n\right\}  ,
\]
while the coefficients of $\sigma:=\sigma^{B}$ are determined from the
corresponding system%
\begin{align*}
c+\sum\limits_{k=1}^{n}a_{k}\varphi_{0}\left(  \frac{r_{j}}{r_{k}}\right)   &
=f\left(  r_{j}\right)  ,\quad\forall j\in\left\{  1,\ldots,n\right\}  ,\\
\sum_{k=1}^{n}a_{k}r_{k}^{-2}  &  =0.
\end{align*}
Next, we use expression (\ref{eq:rbf-k0}) to evaluate $\sigma^{A,B}$
numerically on a ten-times finer grid, \emph{i.e.} at 9 equi-spaced points of
each subinterval $\left[  r_{j},r_{j+1}\right]  $, $j\in\left\{
1,\ldots,n-1\right\}  $. Our numerical approximation of the uniform error
$\left\Vert f-\sigma^{A,B}\right\Vert _{L^{\infty}\left[  1,2\right]  }$ is
then computed as%
\[
E_{h}^{A,B}:=\max\left\{  \left\vert \left(  f-\sigma^{A,B}\right)  \left(
1+h\left(  j-1\right)  +\frac{hl}{10}\right)  \right\vert :j=1,\ldots
,n-1;\ l=1,\ldots,9\right\}  .
\]
This procedure is employed successively for $n-1:=2^{m}$, $m\in\left\{
4,5,\ldots,10\right\}  $. Since the values of $E_{h}^{A,B}$ are expected to
decay proportionally with $h^{\kappa}$, for a positive constant $\kappa$, we
also compute the approximations%
\[
\kappa_{h}^{A,B}:=\log_{2}\left(  E_{h}^{A,B}/E_{h/2}^{A,B}\right)
\approx\kappa.
\]
Tables 1--3 display the numerical results for three data profiles $f$.%

\[%
\begin{array}
[c]{c}%
\begin{array}
[t]{ccccccc}%
\hline
n-1 &  & E_{h}^{A} & \kappa_{h}^{A} &  & E_{h}^{B} & \kappa_{h}^{B}\\
\hline
16 &  & 9.8230\times10^{-4} & 1.5369 &  & 1.0307\times10^{-3} & 1.5709\\
32 &  & 3.3851\times10^{-4} & 1.5670 &  & 3.4691\times10^{-4} & 1.5839\\
64 &  & 1.1425\times10^{-4} & 1.5827 &  & 1.1572\times10^{-4} & 1.5919\\
128 &  & 3.8144\times10^{-5} & 1.5913 &  & 3.8390\times10^{-5} & 1.5959\\
256 &  & 1.2659\times10^{-5} & 1.5956 &  & 1.2700\times10^{-5} & 1.5979\\
512 &  & 4.1886\times10^{-6} & 1.5978 &  & 4.1954\times10^{-6} & 1.5990\\
1024 &  & 1.3838\times10^{-6} & - &  & 1.3849\times10^{-6} & -\\
\hline
&  &  &  &  &  &
\end{array}
\\
\text{Table 1. }f\left(  r\right)  :=f_{\varepsilon}\left(  r\right)  \text{
for }\varepsilon=0.1
\end{array}
\]

\[%
\begin{array}
[c]{c}%
\begin{array}
[t]{ccccccc}%
\hline
n-1 &  & E_{h}^{A} & \kappa_{h}^{A} &  & E_{h}^{B} & \kappa_{h}^{B}\\
\hline
16 &  & 1.7625\times10^{-4} & 1.9414 &  & 1.8257\times10^{-4} & 1.9665\\
32 &  & 4.5890\times10^{-5} & 1.9701 &  & 4.6715\times10^{-5} & 1.9829\\
64 &  & 1.1713\times10^{-5} & 1.9849 &  & 1.1818\times10^{-5} & 1.9913\\
128 &  & 2.9590\times10^{-6} & 1.9924 &  & 2.9724\times10^{-6} & 1.9956\\
256 &  & 7.4366\times10^{-7} & 1.9962 &  & 7.4534\times10^{-7} & 1.9978\\
512 &  & 1.8641\times10^{-7} & 1.9981 &  & 1.8662\times10^{-7} & 1.9989\\
1024 &  & 4.6663\times10^{-8} & - &  & 4.6690\times10^{-8} & -\\
\hline
&  &  &  &  &  &
\end{array}
\\
\text{Table 2. }f\left(  r\right)  =r
\end{array}
\]

\[%
\begin{array}
[c]{c}%
\begin{array}
[t]{ccccccc}%
\hline
n-1 &  & E_{h}^{A} & \kappa_{h}^{A} &  & E_{h}^{B} & \kappa_{h}^{B}\\
\hline
16 &  & 1.5906\times10^{-3} & 2.0093 &  & 1.7113\times10^{-3} & 1.9712\\
32 &  & 3.9510\times10^{-4} & 2.0035 &  & 4.3646\times10^{-4} & 1.9840\\
64 &  & 9.8536\times10^{-5} & 2.0014 &  & 1.1033\times10^{-4} & 1.9916\\
128 &  & 2.4609\times10^{-5} & 2.0006 &  & 2.7744\times10^{-5} & 1.9957\\
256 &  & 6.1496\times10^{-6} & 2.0003 &  & 6.9566\times10^{-6} & 1.9978\\
512 &  & 1.5371\times10^{-6} & 2.0002 &  & 1.7418\times10^{-6} & 1.9989\\
1024 &  & 3.8422\times10^{-7} & - &  & 4.3577\times10^{-7} & -\\
\hline
&  &  &  &  &  &
\end{array}
\\
\text{Table 3. }f\left(  r\right)  =\cos3r
\end{array}
\]

Table 1 refers to the data function $f_{\varepsilon}$ defined by
(\ref{eq:f-eps}) for $\varepsilon=0.1$. In accord with
Theorems~\ref{thm:Linfty} and \ref{thm:sharpness}, the computed values
$\kappa_{h}^{A,B}$ indicate a clear tendency to converge numerically to the
approximation order $\kappa=1.6$. It can be verified that similar numerical
results with $\kappa=\frac{3}{2}+\varepsilon$ hold for other positive values
of $\varepsilon$, showing a direct link between the smoothness/singularity of
the data function and the approximation order $\kappa.$

On the other hand, Tables 2 \& 3 refer to $C^{\infty}$ data
functions\footnote{It can be assumed that these data profiles belong to
$\Lambda_{0}$ after suitable multiplication by a smooth mollifier which takes
the constant value $1$ on $\left[  0,2\right]  $, and which vanishes
identically on an interval $[a,\infty)$ with $a>2$.}. Along with similar
results that can be observed for other smooth data functions, the two tables
suggest the conjecture that, for Beppo Levi $L_{0}$-spline interpolation to
$C^{\infty}$-smooth data profiles $f$, the uniform norm of the error over the
interval $\left[  r_{1},r_{n}\right]  $ decays with (saturation) order
$O\left(  h^{2}\right)  $, as $h\rightarrow0$. The resolution of this
conjecture is likely to require different techniques than those employed to
establish Theorem~\ref{thm:Linfty} and remains a topic for further research.
Note that the $L^{\infty}$-approximation order $2$ was also conjectured by
Johnson \cite{Jo04} for the related problem of thin plate spline interpolation
to scattered data in any planar domain with a smooth boundary.

\subsection{Two compactly supported $L_{0}$-splines}

A remarkable example of a Beppo Levi $L_{0}$-spline has recently been obtained
independently by Johnson \cite{Jo12} in the context of radial basis function
(RBF) methods for multivariable scattered data interpolation. Specifically,
the profile denoted by Johnson as $\eta_{2}$, constructed as part of a family of 
compactly supported and piecewise polyharmonic RBFs, can be expressed as%
\[
\eta_{2}\left(  r\right)  =\frac{4}{3}\left[  \ln2-\varphi_{0}\left(
r\right)  +\varphi_{0}\left(  \frac{r}{2}\right)  \right]  ,\quad r\geq0,
\]
where $\varphi_{0}$ is defined by (\ref{eq:rbf-ann0}). Letting $\rho=\left\{
1,2\right\}  $, it follows that $\eta_{2}\in\mathcal{S}_{0}\left(
\rho\right)  $ and $\eta_{2}$ is supported on the interval $\left[
0,2\right]  $. As stipulated in \cite[Definition~3.9]{Jo12}, the `singular
coefficient' of $\eta_{2}$, \emph{i.e.} the coefficient of $r^{2}\ln r$ in the
linear representation of $\eta_{2}$ as a member of $\mathrm{Ker}\,L_{0}$ over
the interval $\left(  0,1\right)  $, equals $1$.

Note that $k=2$ is the smallest integer $k$ for which there exists a Beppo
Levi $L_{0}$-spline profile in $\mathcal{S}_{0}\left(  \rho\right)  $ with
compact support on $\left[  0,k\right]  $, where $\rho=\left\{  1,2,\ldots
,k\right\}  $. Alternatively, if the same existence problem is considered for
a \emph{non-singular} Beppo Levi $L_{0}$-spline (\emph{i.e.}, having a zero
singular coefficient), then $k=3$ is the smallest integer for which this
problem admits a solution. More precisely, for $\rho=\left\{  1,2,3\right\}
$, there exists, up to a constant factor, a unique $\beta\in\mathcal{S}%
_{0}\left(  \rho\right)  $, which is non-singular and compactly supported on
$\left[  0,3\right]  $. It is given by the formula%
\[
\beta\left(  r\right)  =\frac{1}{5}\left[  27\ln3-32\ln2+5\varphi_{0}\left(
r\right)  -32\varphi_{0}\left(  \frac{r}{2}\right)  +27\varphi_{0}\left(
\frac{r}{3}\right)  \right]  ,\quad r\geq0.
\]

\begin{figure}
%\centering
\includegraphics[height=1.5in,width=2.0in]{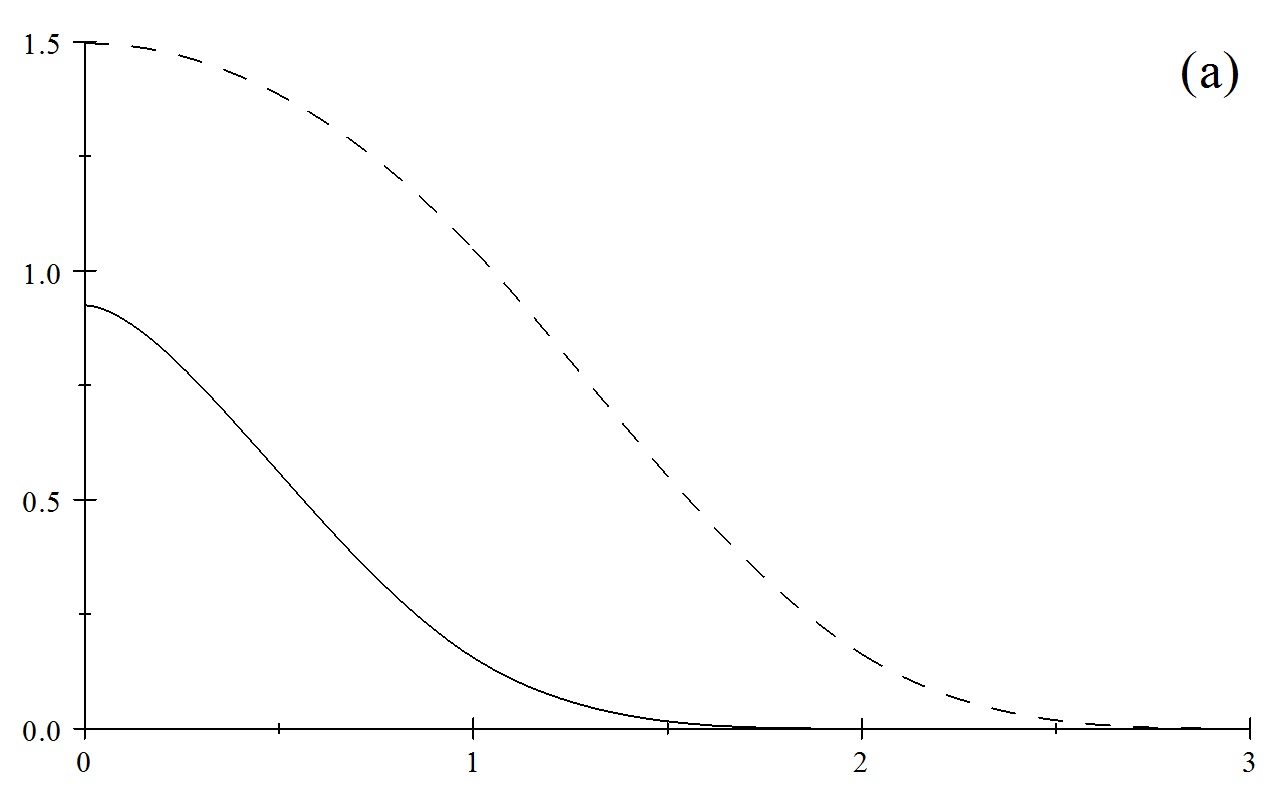}
\hspace{0.5in}
\includegraphics[height=1.5in,width=2.0in]{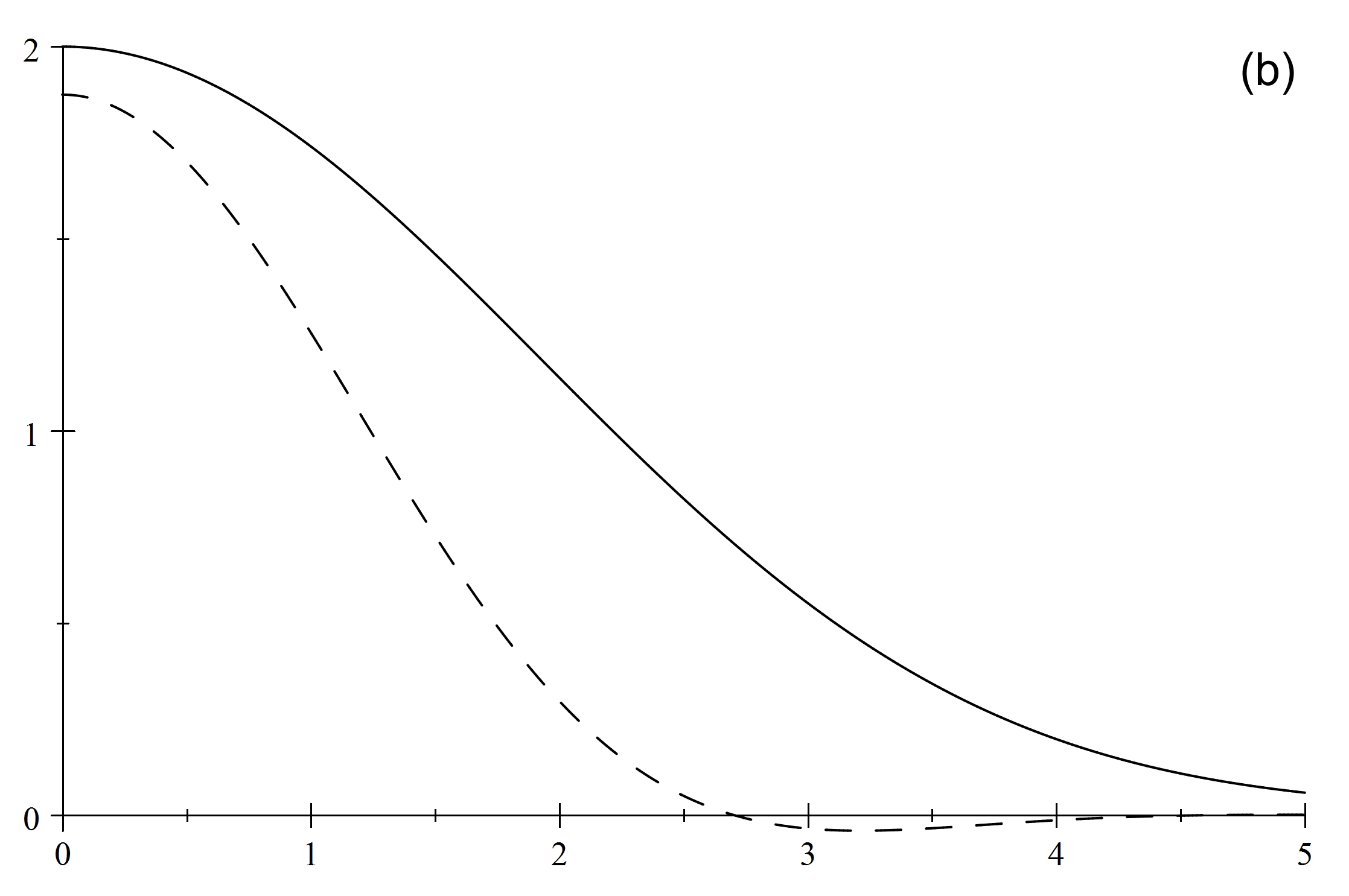}
\caption{{\bf (a)} Plots of $\eta_{2}$ (solid) and $\beta$ (dash). {\bf (b)} Scaled 
plots of $F_2 \eta_{2}$ (solid) and $F_2 \beta$ (dash).}
\label{fig:compact}
\end{figure}

Although not revealed by their plots in Figure~\ref{fig:compact}(a), there is a fundamental
property that recommends $\eta_{2}$ against $\beta$ for use in RBF
interpolation applications. Namely, if $F_{2}\eta_{2}$ denotes the profile of
the Fourier transform of the bivariate radially symmetric extension of
$\eta_{2}$, it is proved in \cite{Jo12} that $F_{2}\eta_{2}$ satisfies a
\emph{Sobolev regularity} condition at $\infty$. In particular, this implies
$\left(  F_{2}\eta_{2}\right)  \left(  t\right)  >0$ for $t\geq0$, hence
$\eta_{2}$ is \emph{positive definite} on $\mathbb{R}^{2}$. In contrast,
using \cite{Jo12} to express $F_{2} \beta$
in terms of the Bessel coefficient $J_{0}$, it is
verified numerically that $F_{2}\beta$
takes both positive and negative values, as can be observed in 
Figure~\ref{fig:compact}(b).

The reader is referred to the monograph \cite{W05} for the
role of positive definiteness and Sobolev regularity in RBF interpolation.
The dilated profile $\eta_2(2\cdot)$ also appears in the recent paper by 
Ward and Unser \cite[Example 2.4]{WU14} as a so-called perturbation of a 
Sobolev spline.

\subsection{Examples of interpolatory $L_0$-splines}

In this subsection, we illustrate some examples of Beppo Levi $L_0$-splines 
that interpolate specific numerical data. Namely, we choose the cases of two 
knots $\rho = \{1,2\}$ in Figure~\ref{fig:twok} and three knots 
$\rho = \{1,2,5\}$ in Figure~\ref{fig:threek}, with corresponding 
interpolation values $\nu_1 = 1$, $\nu_2 = 0$, $\nu_3 = 0.5$. Note that 
this data is also used in Rabut's numerical examples \cite{Rab96}, except 
that in those examples $\nu_3$ has a value that is close to, but different 
from $0.5$.

\begin{figure}
\centering
\includegraphics[height=2.0in,width=3.0in]{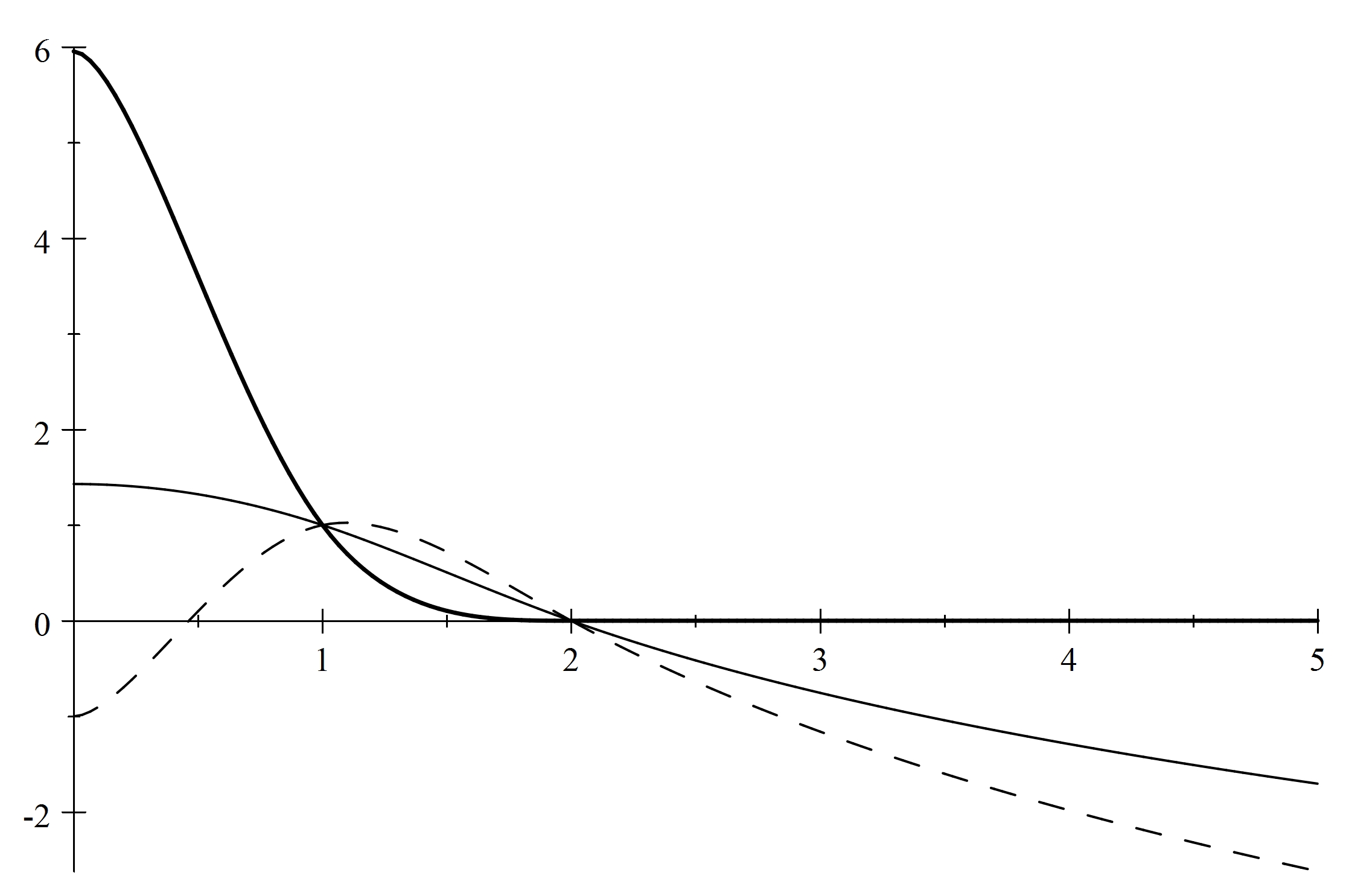}
\caption{Two knots $r_1=1$, $r_2=2$, and interpolation values $\nu_1 = 1$, 
$\nu_2 = 0$. Thick curve: $\sigma^{A}$ for 
$\alpha = \frac {4\ln 2} {5\ln 2 - 3} \approx 5.95$. Thin curve: $\sigma^{B}$. 
Dash curve: $\sigma^{A}$ for $\alpha = -1$.}
\label{fig:twok}
\end{figure}

\begin{figure}
\centering
\includegraphics[height=2.0in,width=3.0in]{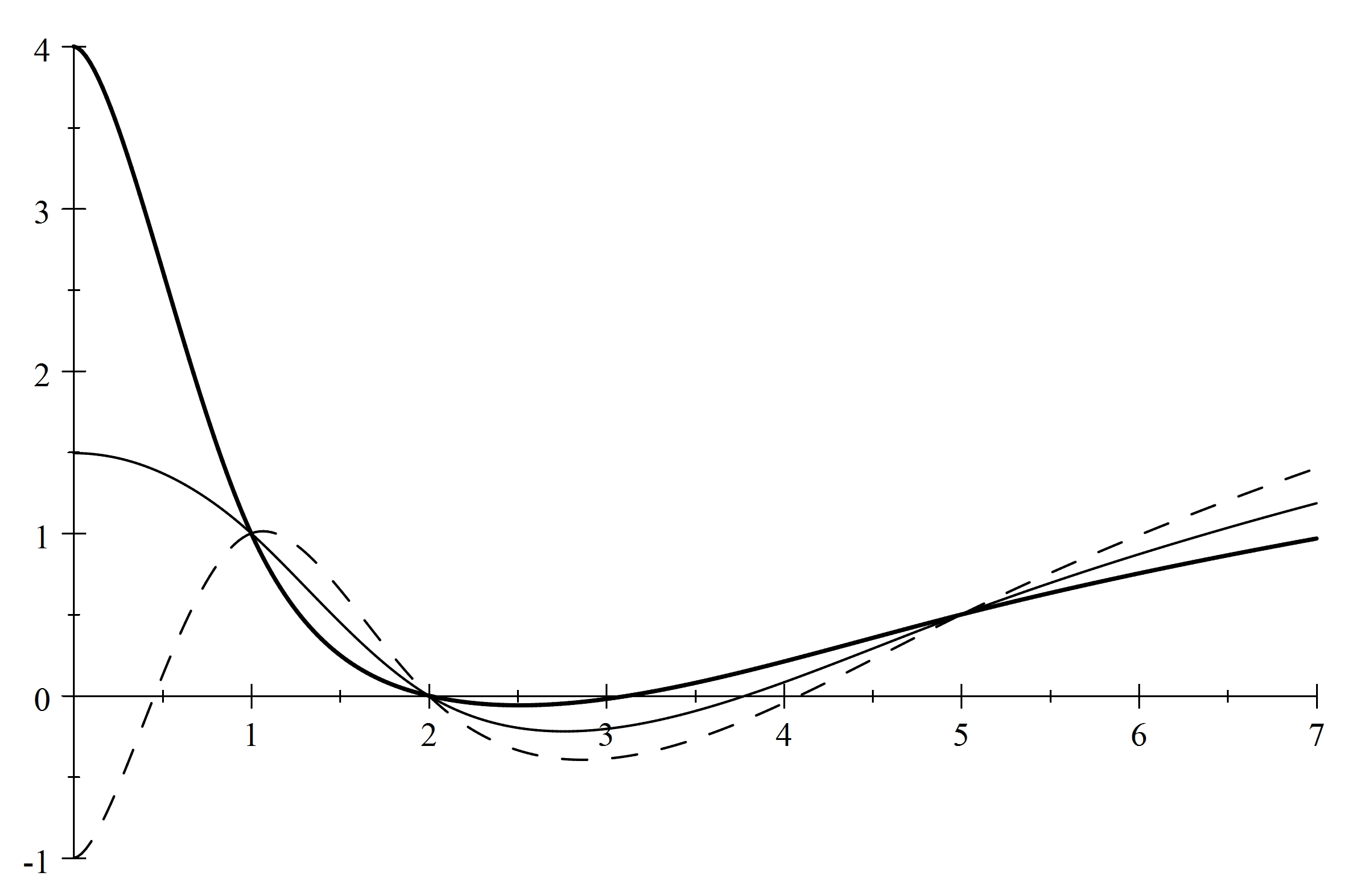}
\caption{Three knots $r_1=1$, $r_2=2$, $r_3=5$, and interpolation values 
$\nu_1 = 1$, $\nu_2 = 0$, $\nu_3=0.5$. Thick curve: $\sigma^{A}$ for 
$\alpha = 4$. Thin curve: $\sigma^{B}$. Dash curve: $\sigma^{A}$ for 
$\alpha = -1$.}
\label{fig:threek}
\end{figure}

Using formula (\ref{eq:rbf-k0}), our Beppo Levi $L_0$-spline interpolants  
are expressed as
\begin{equation}
\sigma\left(  r\right)  = c + a_{1}\varphi_{0}\left( r \right)
+ a_{2}\varphi_{0}\left(\frac{r}{2}\right)
+ a_{3}\varphi_{0}\left(\frac{r}{5}\right), \quad r\geq 0,
\end{equation}
where $a_3 = 0$ in the case of only two knots. In each of the two figures, 
the thin line plot represents the non-singular interpolant $\sigma^{B}$ of 
Theorem~\ref{thm:EU}, whose coefficients satisfy the additional relation 
(\ref{eq:side-B}), \emph{i.e.}\ $a_1 + a_2/4 + a_3/25 = 0$, besides the 
interpolation conditions. The thick and dash line plots correspond to interpolants 
of the type of $\sigma^{A}$ in Theorem~\ref{thm:EU}, for which 
$\sigma^{A}(0) = \alpha$, a prescribed value.

Note that, in the case of two knots considered in Figure~\ref{fig:twok}, 
there exists $\alpha$ for which $\sigma^{A}$ becomes a constant multiple 
of Johnson's profile $\eta_2$ described in the previous subsection. This 
special value $\alpha = \frac {4\ln 2} {5\ln 2 - 3} \approx 5.95$ is used to 
generate the thick curve of Figure~\ref{fig:twok}.

We remark that the thin line graphs of 
Figures~\ref{fig:twok} and \ref{fig:threek} correspond to Rabut's limit case 
in which $R_1 = 0$ and $R_2 = \infty$. In this case, Rabut's observations 
\cite[pp.\ 250--251]{Rab96} on the plot behaviour at $0$ and at $\infty$ 
can be explained, via our Proposition~\ref{prop}, by the form of the 
non-singular Beppo Levi $L_0$-spline $\sigma^{B}$ on the extreme 
subintervals.

\section{Conclusion}

We employed a $L$-spline approach to the problem of minimizing the radial
version (\ref{eq:semin0}) of the Beppo Levi energy integral over the full
semi-axis $\left(  0,\infty\right)  $, subject to interpolation conditions.
This treatment led us to the identification of singular/non-singular solution
profiles, which were expressed as linear combinations of dilates of a basis
function. For $p\geq2$ and data functions from the radial energy space, our
analysis proved the exact $L^{p}$-approximation power $3/2+1/p$ of the
$L$-spline profiles. This result is further used in the sequel paper
\cite{L-spl} to derive a $L^{2}$-error bound for surface interpolation through
curves prescribed on concentric circles.

Additional research is needed to establish the improved approximation order
observed in our numerical experiments for $C^{\infty}$ data functions. Also,
the close connections of Beppo Levi $L$-splines with RBF and polyspline
surfaces motivate the future extension of this work to higher order radially
symmetric piecewise polyharmonic surfaces.
\vspace{5pt} 

\noindent
{\bf Acknowledgements.} 
The author is grateful for the referees' feedback, which has led to an improved 
presentation in the revised version.

\end{document}